\documentclass[11pt]{article}
\usepackage[british]{babel}
\usepackage{amsmath, amsthm, amsfonts, mathrsfs, amsfonts,authblk}
\usepackage{amsthm}
\usepackage{etoolbox}
\usepackage{xcolor}
\usepackage{bm}
\usepackage{amscd}
\usepackage{amssymb}
\usepackage{enumerate} 
\usepackage{enumitem}
\usepackage{multirow}
\usepackage{graphicx}
\usepackage{xypic}
\usepackage{fancyhdr}
\usepackage{mathtools}
\usepackage[a4paper]{geometry}
\usepackage[normalem]{ulem}

\usepackage{titlesec}

\usepackage{accents}
\makeatletter
\newcommand{\sbullet}{%
  \hbox{\fontfamily{lmr}\fontsize{.4\dimexpr(\f@size pt)}{0}\selectfont\textbullet}}

\makeatother

\usepackage[final,bookmarksnumbered=true]{hyperref}
\hypersetup{
     colorlinks = true,
     linkcolor = black,
     anchorcolor = blue,
     citecolor = blue,
     filecolor = blue,
     urlcolor = blue
     }

\newcommand{\cv}[1]{{\color{blue}{#1}}}

\numberwithin{equation}{section}
\numberwithin{table}{section}

\newcommand{\roc}{R} % (global) ring of coefficients
\newcommand{\neec}{O} % neutral element of elliptic curve; used to be \cor{O}
\newcommand{\mfh}{\mathfrak h}
\newcommand{\alg}{K}   % \roc-algebra
\newcommand{\ff}{F}    % finite field
\newcommand{\para}{\delta} % parameter for elliptic curve family; used to be S
\newcommand{\sqrtpara}{{\para}^{\frac{1}{2}}}

\newcommand{\tuB}{\textup{B}}

\newcommand{\tuBd}{\textup{B}^{\bullet}}

\newcommand{\bfe}{{\bf e}}

\newcommand{\bfx}{{\bf x}}
\newcommand{\bfX}{{\bf X}}
\newcommand{\bfy}{{\bf y}}
\newcommand{\bfY}{{\bf Y}}
\newcommand{\mcA}{\mathcal A}
\newcommand{\mcO}{\mathcal O}

\newcommand{\mcS}{\mathcal S}
\newcommand{\bfH}{{\bf H}}
\newcommand{\F}{\mathbb{F}}
\newcommand{\N}{\ensuremath{\mathbb{N}}}
\newcommand{\Z}{\ensuremath{\mathbb{Z}}}
\newcommand{\mbbP}{\ensuremath{\mathbb{P}}}
\newcommand{\Q}{\ensuremath{\mathbb{Q}}}

\newcommand{\C}{\ensuremath{\mathbb{C}}}

\newcommand{\gp}[2]{{\bf G}_{#1}(#2)}

\newcommand{\la}[2]{{\mathfrak g}_{#1}(#2)}
\newcommand{\LA}{\mathfrak g}

\newcommand{\GCD}[1]{\gcd\left( #1 \right)}
\newcommand{\cor}[1]{\mathcal{#1}}

\newcommand{\graffe}[1]{\{#1\}}

\newcommand{\gen}[1]{\langle #1\rangle}

\newcommand{\ul}[1]{\underline{#1}}
\newcommand{\bgr}[1]{\Bigg(#1\Bigg)}

\newcommand{\Fp}{\mathbb{F}_p}

\renewcommand{\L}{\textup{L}}
\renewcommand{\char}{\textup{char}}

\DeclarePairedDelimiter\ceil{\lceil}{\rceil}

\DeclareMathOperator{\HM}{H}
\DeclareMathOperator{\Hes}{Hes}
\DeclareMathOperator{\rk}{rk}

\DeclareMathOperator{\Aut}{Aut} %automorphism group
\DeclareMathOperator{\GL}{GL} %general linear group
\DeclareMathOperator{\PGL}{PGL} %projective linear group
\DeclareMathOperator{\im}{im} %image of a map
\DeclareMathOperator{\hc}{H} %cohomology
\DeclareMathOperator{\Hom}{Hom} %homomorphisms
\DeclareMathOperator{\Cyc}{C} %centralizer
\DeclareMathOperator{\Mat}{Mat} %matrices
\DeclareMathOperator{\id}{id} %identity map
\DeclareMathOperator{\Id}{Id} %identity matrix
\DeclareMathOperator{\Gal}{Gal} %Galois group
\DeclareMathOperator{\tp}{T} %transpose
\DeclareMathOperator{\diag}{diag} %diagonal matrix 
\DeclareMathOperator{\ad}{\Phi}
\DeclareMathOperator{\XE}{\mathcal{X}_\mathit{E}}

\DeclareMathOperator{\cE}{\overline{c}_\mathit{E}}
\DeclareMathOperator{\XEpara}{\mathcal{X}_{\para}}
\DeclareMathOperator{\oXEpara}{\overline{\mathcal{X}_{\para}}}
\DeclareMathOperator{\cpara}{\overline{c}_{\para}}

\def\blankfootnote{\xdef\@thefnmark{}\@footnotetext}

\newtheorem{definition}{Definition}[section]
\newtheorem{lemma}[definition]{Lemma}
\newtheorem{theorem}[definition]{Theorem}
\newtheorem{proposition}[definition]{Proposition}
\newtheorem{corollary}[definition]{Corollary}

\theoremstyle{definition}
\newtheorem*{definition*}{Definition}
\newtheorem*{acknowledgements}{Acknowledgements}

\newenvironment{remark}[1][]{\refstepcounter{definition}\par\medskip
   \noindent \textbf{Remark~\thedefinition. #1} \rmfamily}{\medskip}

 \title{Hessian matrices, automorphisms of $p$-groups, 
   and torsion points of elliptic curves }

\author{Mima Stanojkovski and Christopher Voll \thanks{\noindent{\itshape 2010 Mathematics Subject Classification.}
    20D15, 11G20, 14M12.
    \noindent {\itshape Keywords.} 
    Finite $p$-groups, automorphism groups,
    Heisenberg groups, linear
    symmetric determinantal representations, Hessian matrices, elliptic curves, torsion points.}}

\begin{document}

\maketitle

\abstract{We describe the automorphism groups of finite $p$-groups
  arising naturally via Hessian determinantal representations of
  elliptic curves defined over number fields.  Moreover, we derive
  explicit formulas for the orders of these automorphism groups for
  elliptic curves of $j$-invariant $1728$ given in Weierstrass form.
  We interpret these orders in terms of the numbers of $3$-torsion
  points (or flex points) of the relevant curves over finite
  fields. Our work greatly generalizes and conceptualizes previous
  examples given by du Sautoy and Vaughan-Lee. It explains, in
  particular, why the orders arising in these examples are polynomial
  on Frobenius sets and vary with the primes in a nonquasipolynomial
  manner.  \iffalse We compute the orders of the automorphism groups
  of finite $p$-groups arising naturally via Hessian determinantal
  representations of certain elliptic curves defined over number
  fields. We interpret these orders in terms of the numbers of
  $3$-torsion points (or flex points) of the relevant curves over
  finite fields. Our work greatly generalizes and conceptualizes
  previous examples given by du Sautoy and Vaughan-Lee. It explains,
  in particular, why the orders arising in these examples vary with
  the primes in a ``wild'', viz.\ nonquasipolynomial manner.\fi}

\thispagestyle{empty}

\tableofcontents

\section{Introduction and main results}
In the study of general questions about finite $p$-groups it is
frequently beneficial to focus on groups in natural families. Often,
this affords an additional geometric point of view on the original
group-theoretic questions. When considering, for instance, the members
of a family $({\bf G}(\Fp))_{p \textrm{ prime}}$ of groups of
$\Fp$-rational points of a unipotent group scheme ${\bf G}$, it is of
interest to understand the interplay between properties of the
abstract groups ${\bf G}(\Fp)$ with the structure of the group
scheme~${\bf G}$. Specifically, geometric insights into the
automorphism group $\Aut({\bf G})$ of ${\bf G}$ translate into uniform
statements about the automorphism groups $\Aut({\bf G}(\F_p))$.

%In this paper we use this approach to compute the orders of the
%automorphism groups of certain finite $p$-groups and interpret them in
%terms of the arithmetic of elliptic curves. The groups we consider
%arise naturally in the framework of a well-known, general construction
%of groups (and Lie algebras) from bilinear forms or, equivalently,
%matrices of linear forms. The specific $p$-groups we focus on come
%from a classical algebro-geometric construction of linear symmetric
%determinantal representations of elliptic curves via Hessian matrices.

In this paper we use this approach to compute the orders of the
automorphism groups of groups and Lie algebras ${\bf G}_{\tuB}(\ff)$
resp.\ $\LA_{\tuB}(\ff)$ defined in terms of a matrix of linear
forms~$\tuB$, where $F$ is a finite field of odd characteristic. In
the case that $\tuB$ is a Hessian determinantal representation of an
elliptic curve, we give an explicit formula for
$|\Aut(\LA_{\tuB}(\ff))|$; up to a scalar, this formula also gives
$|\Aut({\bf G}_{\tuB}(\ff))|$. We consequently apply this result to a
parametrized family of elliptic curves and interpret this formula in
terms of arithmetic invariants of the relevant curves;
cf.\ Theorem~\ref{th:main.para}.  Notwithstanding the fact that our
main results are formulated for finite fields, the underpinning
structural analysis applies to a larger class of fields, potentially
even more general rings.

The following theorem is a condensed summary of the main results of
this paper. Throughout we denote, given an elliptic curve
$E$ and $n\in\N$, by $E[n]$ the $n$-torsion points of~$E$. 

\begin{theorem}\label{th:intro} Let $E$ be an elliptic curve
  over $\Q$ and let $\ff$ be a finite field of odd characteristic $p$
  over which $E$ has good reduction. Write, moreover,
  $\Aut_{\neec}(E)$ for the automorphism group of the elliptic curve
  $E$ and assume that $|E[2](\ff)|=4$. Then there exist groups ${\bf
    G}_1(\ff)$, ${\bf G}_2(\ff)$, and ${\bf G}_3(\ff)$ such that the
  following hold:
\begin{enumerate}[label=$(\arabic*)$]
 \item each ${\bf G}_i(\ff)$ is a group of order $|\ff|^9$, exponent $p$, and nilpotency class $2$;  
 \item for each $i=1,2,3$, there exists $T_i\leq E\rtimes\Aut_\neec(E)$ such that 
 \[
|\Aut({\bf G}_i(\ff))|=|\ff|^{18}\cdot|\GL_2(\ff)|\cdot  |T_i(\ff)|\cdot |\Gal(\ff/\F_p)|.
\]
\end{enumerate}
Moreover, if $\para\in\ff\setminus\graffe{0}$ is such that $E=E_\para$
is given by $y^2=x^3-\para x$ over $F$, then
%$$T_i(E_{\para})=E_{\para}[3]\rtimes\Aut_{\neec}(E_{\para})[\lceil 4/i\rceil ]$$
$$|T_{i}(F)|=|E_{\para}[3](F)| \cdot \gcd(|F|-1,[\lceil 4/i\rceil ])$$ 
and, for $i\neq j$, the groups ${\bf G}_i(F)$ and ${\bf G}_j(F)$ are isomorphic if and only if $\{i,j\}=\{2,3\}$ and  $p\equiv 1\bmod 4$.
%Any two groups associated to distinct $\para$'s modulo $p$ are non-isomorphic. 
Any two groups associated with distinct values of $\para$ are
non-isomorphic. 
\end{theorem}

\noindent We remark that, at least apart from characteristic~$3$,
Theorem~\ref{th:intro} covers all elliptic curves over $\Q$ with
$j$-invariant~$1728$. It also implies that, for certain values
of~$\para$, the function $$p \mapsto |\Aut(\gp{i}{\Fp})|$$ is
polynomial on Frobenius sets of primes; see
Section~\ref{subsubsec:POFS} for definitions and details.

In the remainder of the introduction we progressively illustrate some
of the paper's ideas. These include explicit constructions,
motivation, and broader context. We will prove Theorem~\ref{th:intro}
in Section~\ref{subsec:iso}.

\subsection{A Hessian determinantal representation} Let $y_1$, $y_2$, and
$y_3$ be independent variables and define the polynomial
\begin{equation}\label{def:f1}
  f_1(y_1,y_2,y_3) = y_1^3 - y_1y_3^2 - y_2^2y_3\in\Z[y_1,y_2,y_3].
\end{equation}
The \emph{Hessian} (\emph{polynomial}) $\Hes(f_1)$ of $f_1$ is the
determinant $\det(\HM(f_1))$ of the \emph{Hessian matrix}
$\HM(f_1) = \left( \frac{\partial^2 f_1}{\partial y_i \partial
    y_j}\right)_{ij}\in\Mat_3(\Z[y_1,y_2,y_3])$
associated with~$f_1$. The Hessian matrix of $\Hes(f_1)$, in turn, gives

\begin{equation}\label{def:B11.new}
  \tuB_{1,1}(y_1,y_2,y_3)\stackrel{\textup{def}}{=} \frac{\HM(\Hes(f_1))}{48} = \begin{pmatrix}
    y_3 & -y_2 & y_1 \\
    -y_2 & -y_1 & 0 \\
    y_1 & 0 & y_3
  \end{pmatrix} \in \Mat_3(\Z[y_1,y_2,y_3]).
\end{equation}
We note that the equation
\begin{equation}\label{eq:hess.1}
  48^3 f_1 = \Hes(\Hes(f_1)) = 48^3 \det(\tuB_{1,1})
\end{equation}
is satisfied. The matrix $\tuB_{1,1}$ is thus a \emph{linear symmetric
  determinantal representation} of the polynomial~$f_1$.

\subsection{Finite $p$-groups from matrices of linear forms}\label{subsec:intro.dSVL}
For every prime $p$, the representation~\eqref{def:B11.new} gives rise
to a finite $p$-group $\gp{1,1}{\Fp}$ via the following presentation:
\begin{align}\label{def:pres.G11}
  \gp{1,1}{\Fp} = \gen{
  e_1,e_2,e_3,&\,f_1,f_2,f_3,\,g_1,g_2,g_3\ \mid  \\
              & \textup{class $2$, exponent $p$, } \gen{e_1,e_2,e_3} \textup{ and }  \gen{f_1,f_2,f_3} \textup{ abelian, } \nonumber  \\
              & [e_1,f_1]=[e_3,f_3]=g_3, \, \
                [e_1,f_2]=[e_2,f_1]= g_2^{-1},\nonumber\\
              & [e_1,f_3]=[e_2,f_2]^{-1} = [e_3,f_1]=g_1,\ %\nonumber\\
  [e_2,f_3]=[e_3,f_2]=1, \nonumber 
  }.
\end{align}
Indeed, the group relations $[e_1,f_3]=[e_2,f_2]^{-1} = [e_3,f_1]=g_1$
reflect the linear relations $b_{13} = -b_{22} = b_{31} = y_1$ among
the entries of the matrix $\tuB_{1,1}=(b_{ij})$; the group relation
$[e_2,f_3]=1$ reflects the linear relation $b_{23}=0$ etc. (This
ad-hoc definition is a special case of a general construction recalled
in~Section~\ref{subsec:groups}.) If $p$ is odd, then $\gp{1,1}{\Fp}$
is a group of order~$p^9$, exponent~$p$, and nilpotency class~$2$.

In \cite{dS+VL}, du Sautoy and Vaughan-Lee computed the orders of the
automorphism groups of the groups $\gp{1,1}{\Fp}$, for primes
$p>3$. This was a major step towards their aim of showing that {the
  numbers of immediate descendants of these groups of order $p^{10}$
  and exponent $p$ are not a PORC-function of the primes;
  see~Theorem~\ref{th:main.dSVL} and Sections~\ref{subsubsec:POFS}
  and~\ref{subsubsec:PORC}. (The groups $G_p$ defined in \cite{dS+VL}
  can easily be seen to be isomorphic to the groups $\gp{1,1}{\Fp}$.)
  The purpose of the present paper is twofold: first, to generalize
  these computations to a larger class of groups (or rather, group
  schemes); second, to give a conceptual interpretation of the
  computations in \cite{dS+VL} in terms of Hessian matrices and
  torsion points of elliptic curves. We reach both aims in
  Theorem~\ref{th:main.para}.

As we now explain, Theorem~\ref{th:main.para} provides new insight,
even where it reproduces old results. To see this, note that
$f_1 = \det(\tuB_{1,1})$ defines the elliptic curve
$$E_1:\; y^2=x^3-x$$ over~$\Q$. Recall that we denote, given an
elliptic curve $E$ over a field $\ff$, by $E[3]$ the group of
$3$-torsion points of~$E$. The collection of its $\ff$-rational points
$E[3](\ff)$ is then isomorphic to a subgroup of $\Z/(3)\times
\Z/(3)$. The following is a special case of our
Theorem~\ref{th:main.para}.

\begin{theorem}[du Sautoy--Vaughan-Lee]\label{th:main.dSVL}
  Assume that $p>3$. Then the following holds:
  \begin{equation*}
    |\Aut(\gp{1,1}{\Fp})|=
    \GCD{p-1, 4} |\GL_2(\Fp)|  p^{18} \cdot |E_1[3](\Fp)|.
    \end{equation*}
  \end{theorem}
  
  \noindent
  We write $\mu_4$ for the group scheme of $4$th roots of unity. We remark
  that the factor 
  $$\GCD{p-1, 4} = |\mu_4(\Fp)| = 3 + \left( \frac{-1}{p} \right)\in\{2,4\}$$
  depends only on the equivalence class of $p$~modulo~$4$. In stark contrast,
  it follows from the analysis of \cite{dS+VL} that
\begin{equation}
  |E_1[3](\Fp)| = \begin{cases} 9,
    &\textup{if $p\equiv 1 \bmod 12$ and there exist solutions in } \Fp\times
    \Fp\\& \textup{to $y^2=x^3-x$ and $x^4 + 6x^2-3=0$},\\ 3, &\textup{ if }p \equiv -1 \bmod 12,\\ 1,
    &\textup{otherwise.} \end{cases}\label{eq:aut}
\end{equation}

\noindent
This case distinction is not constant on primes with fixed residue
class modulo any modulus. A concise, explicit description of
$|\Aut(\gp{1,1}{\Fp})|$, which is implicit in \cite{dS+VL}, is given
in~\cite[Sec.~5.2]{Vaughan-Lee/12}. Via our
Lemma~\ref{lem:dSVL.torsion.points}, one can recover \eqref{eq:aut}
from it.  In fact, \eqref{eq:aut} uses du Sautoy and Vaughan-Lee's
formulation in terms of the solvability of the quartic $x^4 + 6x^2-3$
among the $\Fp$-rational points of~$E_1$.  One of the main
contributions of the present article is to connect this condition with
the structure of the group of $3$-torsion points of $E_1$, affording
an arithmetic interpretation.  We remark that the $3$-torsion points
of $E_1$ are exactly the flex points of $E_1$, i.e.\ the points which
also annihilate $\Hes(f_1) = 8(y_3^3 + 3 y_1^2 y_3 - 3y_1y_2^2)$; see
Lemma~\ref{lemma:3torEQ.para}.
  
\subsection{Generalization 1: further Hessian representations}\label{sec:gen.hess}
Our first generalization of Theorem~\ref{th:main.dSVL} is owed to the
fact that the Hessian matrix $\tuB_{1,1}$ in \eqref{def:B11.new} has
two natural ``siblings''.

Indeed, let $f\in\C[y_1,y_2,y_3]$ be a homogeneous cubic polynomial
defining a smooth projective curve. It is a well-known
algebro-geometric fact that the \emph{Hessian equation}
\begin{equation}\label{eq:gen.hess}
  \alpha f = \Hes(\beta f  + \Hes(f))
\end{equation}
has exactly three solutions $(\alpha,\beta)\in\C^2$, yielding pairwise
inequivalent linear symmetric determinantal representations of $f$
over~$\C$; see also Section~\ref{subsec:gen.para}. In fact, \emph{any}
linear symmetric representation of $f$ is equivalent to one arising in
this way. For modern accounts of this classical construction, which is
presumably due to Hesse~\cite{Hesse/44}, see, for instance,
\cite[Prop.~5]{PlaumannSturmfelsVinzant/12},
\cite[Sec.~5]{BuckleyPlestenjak/18}, and \cite[Ch.~II.2]{Harris/79}.

One of the three solutions of the Hessian equation for the specific polynomial
$f=f_1$ defined in~\eqref{def:f1} is $(\alpha,\beta) = (48^3,0)$;
cf.~\eqref{eq:hess.1}. A short computation yields the two others, viz.\
$(\alpha,\beta) = (4(48)^3,\pm 24)$. This leads us to
complement~\eqref{def:B11.new} as follows:
\begin{alignat}{2}\label{def:Bi1}
  \tuB_{1,1}(\bfy)&\stackrel{\textup{def}}{=}& \frac{\HM(\Hes(f_1))}{48}
  =& \begin{pmatrix}
    y_3 & -y_2 & y_1 \\
    -y_2 & -y_1 & 0 \\
    y_1 & 0 & y_3
  \end{pmatrix},
\\  \tuB_{2,1}(\bfy)&\stackrel{\textup{def}}{=}& \frac{\HM(+24f_1 + \Hes(f_1))}{48}=&
\begin{pmatrix}
  3y_1 + y_3 & -y_2 & y_1 -  y_3 \\
  -y_2 & -y_1-  y_3 & - y_2& \\
  y_1 -  y_3 & - y_2 & - y_1 +  y_3
\end{pmatrix}, \nonumber \\
  \tuB_{3,1}(\bfy) & \stackrel{\textup{def}}{=}&\ \frac{\HM(-24 f_1 + \Hes(f_1))}{48}=&
\begin{pmatrix}
  -3y_1 + y_3 & -y_2 & y_1 +  y_3 \\
  -y_2 & -y_1 +  y_3 &  y_2& \\
  y_1 +  y_3 &  y_2 & y_1 +  y_3
\end{pmatrix}.\nonumber %\in\Mat_3(\Z(t)[y_1,y_2,y_3]).
\end{alignat}
The identities
$4f_1 = 4\det(\tuB_{1,1}) = \det(\tuB_{2,1})= \det(\tuB_{3,1})$ are
easily verified.  \iffalse  
\cv{\begin{align}\label{def:Bi1.old}
    \tuB_{1,1}(y_1,y_2,y_3)&=\begin{pmatrix}
      y_3 & -y_1 & y_2 \\
      -y_1 & y_3 & 0 \\
      y_2 & 0 & -y_1
\end{pmatrix}, \nonumber\\ 
\tuB_{2,1}(y_1,y_2,y_3)&=
\begin{pmatrix}
y_1-y_3 & y_1-y_3 & -y_2 \\
y_1-y_3 & -3y_1-y_3 & y_2 \\
-y_2 & y_2 & y_1+y_3
\end{pmatrix}, \nonumber \\
\tuB_{3,1}(y_1,y_2,y_3) &=
\begin{pmatrix}
-y_1-y_3 & y_1+y_3 & y_2 \\
y_1+y_3 & 3y_1-y_3 & y_2 \\
y_2 & y_2 & y_1-y_3
\end{pmatrix} \in\Mat_3(\Z[y_1,y_2,y_3]). 
    \end{align}}
  \fi
  Straightforward generalizations of the presentation~\eqref{def:pres.G11}
  yield, for every odd prime $p$, groups $\gp{i,1}{\Fp}$ of order~$p^9$,
  exponent~$p$, and nilpotency class~$2$. The following generalizes
  Theorem~\ref{th:main.dSVL}.

\begin{theorem}\label{th:main}
  Assume that $p$ is odd and let $i\in\{1,2,3\}$. Then the following holds:
  \begin{align*}
    |\Aut(\gp{i,1}{\Fp})|   & =\GCD{p-1, \ceil*{4/i}}|\GL_2(\Fp)|  p^{18}\cdot |E_1[3](\Fp)|.
    \end{align*}
  \end{theorem}
  \noindent 
  Note that the factor $\GCD{p-1, \ceil*{4/i}}$ is constant equal to $2$ for
  $i\in\{2,3\}$.  That the cases $i=1,2,3$ are not entirely symmetric suggests
  a corresponding asymmetry in the three solutions to the Hessian
  equation~\eqref{eq:gen.hess}.  The geometric fact
  (\cite[Thm.~1(1)]{RaviTri14}) that they correspond to the nontrivial
  $2$-torsion points of $E_1$ may help to shed light on this phenomenon.

\subsection{Generalization 2: base extension and families of
    elliptic curves}\label{subsec:gen.para}
  We generalize Theorem~\ref{th:main.dSVL} further, simultaneously in two
  directions. To this end, let $\para$ be a nonzero integer and define
\begin{equation}\label{def:fS}
  f_\para(y_1,y_2,y_3) = y_1^3 - \para y_1y_3^2 - \para y_2^2y_3\in\Z[y_1,y_2,y_3],
\end{equation}
generalizing~\eqref{def:f1}. Denoting by $\sqrtpara$ a fixed square
root of $\para$, a short computation yields that the Hessian
equation~\eqref{eq:gen.hess} has the three solutions
$((48\para^2)^3,0)$ and
$(4(48\para^2)^3,\pm24 \para^{\frac{3}{2}})$. Generalizing the matrices of
linear forms defined in~\eqref{def:Bi1}, we set, in
$\Mat_3(\Z_{\para}[\sqrtpara][y_1,y_2,y_3])$,

\begin{alignat}{2}\label{def:BiS}
  \tuB_{1,\para}(\bfy)&\stackrel{\textup{def}}{=}&
  \frac{\HM(\Hes(f_{\para}))}{48 \para^2} & = \begin{pmatrix}
    y_3 & - y_2 & y_1\\
    -y_2 & - y_1 & 0 \\
    y_1 & 0 & \para y_3
  \end{pmatrix}, \\
  \tuB_{2,\para}(\bfy)&\stackrel{\textup{def}}{=} & \ \frac{\HM(+24
    \para^{\frac{3}{2}} f_{\para}+ \Hes(f_{\para}))}{48 \para^2} &=
\begin{pmatrix}
  3 \para^{-\frac{1}{2}} y_1 + y_3 & -y_2 & y_1 - \sqrtpara y_3 \\
  -y_2 & -y_1- \sqrtpara y_3 & -\sqrtpara y_2& \\
  y_1 - \sqrtpara y_3 & -\sqrtpara y_2 & -\sqrtpara y_1 + \para y_3
\end{pmatrix}, \nonumber \\
\tuB_{3,\para}(\bfy) &\stackrel{\textup{def}}{=} & \ \frac{\HM(-24
  \para^{\frac{3}{2}} f_{\para} + \Hes(f_{\para}))}{48 \para^2} &=
\begin{pmatrix}
  -3 \para^{-\frac{1}{2}} y_1 + y_3 & -y_2 & y_1 + \sqrtpara y_3 \\
  -y_2 & -y_1 + \sqrtpara y_3 & \sqrtpara y_2& \\
  y_1 + \sqrtpara y_3 & \sqrtpara y_2 & \sqrtpara y_1 + \para y_3
\end{pmatrix},\nonumber
\end{alignat}
where $\Z_{\para}$ denotes the localization of $\Z$ with respect to the
$\para$-powers.  \iffalse
\begin{align}\label{def:BiS.old}
    \tuB_{1,\para}(y_1,y_2,y_3)&=\begin{pmatrix}
      y_3 & -y_1 & y_2 \\
      -y_1 & S y_3 & 0 \\
      y_2 & 0 & -y_1
  \end{pmatrix}, \nonumber\\
  \tuB_{2,\para}(y_1,y_2,y_3)&=
\begin{pmatrix}
  \sqrtpara y_1- S y_3 & y_1- \sqrtpara y_3 & - \sqrtpara y_2 \\
  y_1 -\sqrtpara y_3 & -3 \para^{-1/2}y_1-y_3 & y_2 \\
  -\sqrtpara y_2 & y_2 & y_1+\sqrtpara y_3
\end{pmatrix}, \nonumber \\
  \tuB_{3,\para}(y_1,y_2,y_3) &=
                            \begin{pmatrix}
                              -\sqrtpara y_1- S y_3 & y_1+ \sqrtpara y_3 & \sqrtpara y_2 \\
                              y_1+\sqrtpara y_3 & 3\para^{-1/2}y_1-y_3 & y_2 \\
                              \sqrtpara y_2 & y_2 & y_1-\sqrtpara y_3
                            \end{pmatrix}
                                                 \in\Mat_3(\Z(t)[y_1,y_2,y_3]).
\end{align}}
\fi
The identities
$4f_{\para} = 4\det(\tuB_{1,\para}) = \det(\tuB_{2,\para})= \det(\tuB_{3,\para})$ are
easily verified. Clearly, interchanging $\para^{\frac{1}{2}}$ with its negative just interchanges~$\tuB_{2,\para}$ and~$\tuB_{3,\para}$.

We also avail ourselves of a general, well-known construction---recalled in
detail in Section~\ref{subsec:groups}---which associates, in particular, to
each of the $\tuB_{i,\para}$ a unipotent group scheme
${\bf G}_{i,\para} = {\bf G}_{\tuB_{i,\para}}$ defined
over~$\Z_{\para}[\sqrtpara]$. For a finite field $\ff$ in which $\para$ is
nonzero and has a (fixed) square root, we denote by $\gp{i,\para}{\ff}$ the
group of $\ff$-rational points of ${\bf G}_{i,\para}$. These groups are
$p$-groups of order $|\ff|^9$ and nilpotency class~$2$.  We also assume for
the rest of the paper that the characteristic of $\ff$ be odd so that they
have exponent~$p$. The groups $\gp{i,\para}{\ff}$ have a number of alternative
descriptions, including one as generalized Heisenberg groups; cf.\
Section~\ref{subsubsec:heisenberg}. The following is our first main result.

\begin{theorem}\label{th:main.para}
  Let $\para\in\Z$ and let $\ff$ be a finite field of characteristic $p$ not
  dividing~$2\para$ and cardinality $p^f$ in which $\para$ has a fixed square
  root. For $i\in\{1,2,3\}$, the following holds:
  \begin{align*}
    |\Aut(\gp{i,\para}{\ff})| 
    &= f\cdot\GCD{|F|-1,\ceil*{4/i}}
    |\GL_2(\ff)|\cdot  |\ff|^{18} \cdot |E_{\para}[3](\ff)|.
  \end{align*}  
\end{theorem}

\begin{remark}\label{rmk:3-torsionDistr} The quantity
  $|E_\para[3](\ff)|$ equals $1$ unless
  $p=\char(\ff)\equiv \pm 1 \bmod 12$. Indeed, it enumerates (the
  curve's point at infinity and) the simultaneous solutions to the
  equations $b^2 = a^3 - \para^{-1}a$ and
  $3\para^2a^4 - 6\para a^2 - 1=0$; cf.\
  Lemma~\ref{lemma:3torEQ.para}. The biquadratic equation is solvable
  in $\ff$ only if its discriminant $48 \para^2$ is a square in
  $\ff$. By quadratic reciprocity, this happens if and only if
  $p \equiv \pm 1 \bmod 12$. Following the same reasoning as in
  \cite[\S~2.2]{dS+VL}, one can show that $|E_\para[3](\ff)|=3$ if and
  only if $p\equiv -1\bmod 12$ and that, if $p\equiv 1 \bmod 12$, then
  $|E_\para[3](F)|\in\graffe{1,9}$. Describing the fibres of the maps
  $p\mapsto|E_\para[3](\Fp)|$ explicitly seems to be, in general, a
  difficult problem; see also Section~\ref{subsubsec:POFS}.
 \end{remark}

\noindent
For any $\para\in\Z\setminus\{0\}$, the matrices $\tuB_{i,\para}(\bfy)$
defined in \eqref{def:BiS} are inequivalent in the following geometric sense:
for any $i,j\in\{1,2,3\}$ with $i\neq j$, there does not exist $U\in\GL_3(\C)$
such that $U \tuB_{i,\para} U^{\tp} = \tuB_{j,\para}$; see
\cite[Prop.~5]{PlaumannSturmfelsVinzant/12}.  In light of this, it is natural
to ask about isomorphisms between groups of the form $\gp{i,\para}{\ff}$. The
next result settles this question, at least for finite prime
fields.

\begin{theorem}\label{th:iso}
  Let $i,j\in\graffe{1,2,3}$ and let $\para,\para'\in\Z\setminus\graffe{0}$,
  and assume that $p$ is a prime not dividing $2\para\para'$. Then
  $\gp{i,\para}{\Fp}\cong\gp{j,\para'}{\Fp}$ if and only if $\para=\para'$ in
  $\Fp$ and either
\begin{enumerate}[label=$(\arabic*)$]
\item $i=j$ and, if $i\in\graffe{2,3}$, then either
\begin{enumerate}[label=$(1.\alph*)$]
  \item $\sqrtpara=\para'^{\frac{1}{2}}$ or
  \item $\sqrtpara=-\para'^{\frac{1}{2}}$ and $p \equiv 1 \bmod 4$ or
\end{enumerate}
\item $\graffe{i,j}=\graffe{2,3}$ and either
\begin{enumerate}[label=$(2.\alph*)$]
  \item $\sqrtpara=-\para'^{\frac{1}{2}}$ or
  \item $\sqrtpara=\para'^{\frac{1}{2}}$ and $p \equiv 1 \bmod 4$.
\end{enumerate}
\end{enumerate}
\end{theorem}

\subsection{Context and related research}
\subsubsection{Elliptic curves, torsion points, and linear
  automorphisms}\label{subsubsec:trans}
%While Theorem~\ref{th:main.para} is formulated for a specific family of
%elliptic curves, it relies on the much more general
%Corollary~\ref{cor:formula}. The latter  
Theorem~\ref{th:intro} describes the number of automorphisms of a
group of the form ${\bf G}_{\tuB}(\ff)$, in case that the matrix
$\tuB$ is a Hessian determinantal representation of an elliptic
curve}, in terms of the numbers of such automorphisms that induce
  automorphisms of the elliptic curve; see Sections~\ref{sec:curves}
  and~\ref{sec:3} for details.  That $3$-torsion points of elliptic
  curves play a role in this context is, at least \emph{a posteriori},
  not entirely surprising. Indeed, for $E=E_{\para}$, our proof of
  Theorem~\ref{th:intro} relies on the realisability of translations
  by $3$-torsion points of the elliptic curve $E_\para$ by elements
  of~$\PGL_3(\ff)$, viz.\ $(\Aut~\mathbb{P}^2)(\ff)$. We first show
  that the only translations of $E_{\para}$ that lift to $\PGL_3(\ff)$
  are those coming from $3$-torsion points; see
  Lemma~\ref{lemma:oneil}. The actual realisability of these as linear
  transformations is implicit in the proof of
  Theorem~\ref{th:main.para}. A similar phenomenon is discussed in
  \cite[Rem.~B]{O'Neil/01}: under suitable assumptions, translations
  by $n$-torsion points of genus one curves embedded into
  $\mathbb{P}^{n-1}$ are induced by linear automorphisms
  of~$\mathbb{P}^{n-1}$.  For an elliptic curve with point at infinity
  $\neec$, this translates to the divisor $n\neec$ being \emph{very
    ample} and so, in the case of planar elliptic curves, to $3\neec$
  being very ample. We would like to understand whether the theory of
  divisors can be used also to show that all other translations are
  not linear.

With a broader outlook, it is of course of interest to consider groups arising
from (not necessarily symmetric) determinantal representations of curves or
other algebraic varieties. For the time being, however, a general theory
connecting geometric invariants of determinantal varieties with algebraic
invariants of finite $p$-groups associated with these varieties'
representations eludes us.

\subsubsection{Quasipolynomiality vs.\ polynomiality on Frobenius
  sets}\label{subsubsec:POFS} Let $\Pi$ denote the set of rational
primes. It was shown in \cite{dS+VL} that the function $\Pi\rightarrow
\Z$, $p \mapsto |\Aut(\gp{1,1}{\Fp})|$ is not quasipolynomial (or
Polynomial On Residue Classes (PORC));
cf.\ \cite[Sec.~4.4]{Sta12}. Our result Theorem~\ref{th:main.para} has
the following consequence.

\begin{corollary}\label{cor:non-PORC}
  Let $i\in\{1,2,3\}$ and assume that $\para$ is the fourth power of an integer. Then
  the function $p\mapsto|\Aut(\gp{i,\para}{\Fp})|$ is not PORC.
 \end{corollary}
 
 \noindent
 Similar results hold for general $\para$ if and only if the function
 $p \mapsto |E_\para[3](\Fp)|$ is not constant on sets of primes with
 fixed residue class modulo some integer; cf.\ \eqref{eq:aut}.

 Quasipolynomiality is quite a restrictive property for a counting
 function. Where it fails, it is natural to look for other
 arithmetically defined patterns in the variation with the
 primes. Recall, e.g.\ from \cite{BMS/17} or \cite{Lagarias/83}, that
 a set of primes is a \emph{Frobenius set} if it is a finite Boolean
 combination of sets of primes defined by the solvability of
 polynomial congruences. A function $f:\Pi\rightarrow \Z$ is
 \emph{Polynomial On Frobenius Sets} (\emph{POFS}) if there exist a
 positive integer $N$, Frobenius sets $\Pi_1,\dots,\Pi_N$
 partitioning~$\Pi$, and polynomials $f_1,\dots,f_N\in\Z[T]$ such that
 the following holds:
 $$p \in \Pi_j \Longleftrightarrow f(p) = f_j(p).$$
 Our Theorem~\ref{th:main.para} implies, for instance, the following.
 \begin{corollary}\label{cor:POFS}
   Let $i\in\{1,2,3\}$ and assume that $\para$ is the square of an
   integer if $i\neq 1$. Then the function
   $p \mapsto |\Aut(\gp{i,\para}{\Fp})|$ is POFS.
 \end{corollary}
 
 \noindent
% (The assumption on $\para$ is made for notational convenience. It
% ensures that the group schemes ${\bf G}_{i,\para}$ have
% $\Fp$-rational points for all rational primes if~$i\in\{2,3\}$.)
In fact, for $\para=1$ one may take $N=4$
 for $i=1$ and $N=3$ for $i\in\{2,3\}$.

 Corollary~\ref{cor:POFS} invites a comparison of
 Theorem~\ref{th:main.para} with a result by Bardestani,
 Mallahi-Karai, and Salmasian. Indeed, \cite[Thm.~2.4]{BMS/17}
 establishes---in notation closer to the current paper---that, for a
 unipotent group scheme ${\bf G}$ defined over~$\Q$, the faithful
 dimension of the $p$-groups ${\bf G}(\Fp)$ (viz.\ the smallest $n$
 such that ${\bf G}(\Fp)$ embeds into $\GL_n(\C)$) defines a POFS
 function.  For a number of recent related quasipolynomiality results,
 see~\cite{EVL/20}.

 \subsubsection{Automorphism groups and immediate
   descendants}\label{subsubsec:PORC}
 Du Sautoy and Vaughan-Lee embedded their discussion of the
 automorphism groups of the groups $\gp{1,1}{\Fp}$ in a study of the
 immediate descendants of order $p^{10}$ and exponent~$p$ of these
 groups of order~$p^{9}$.  In fact, their paper's main result is the
 statement that the numbers of these descendants is not PORC as a
 function of~$p$. Theorem~\ref{th:main.dSVL} allows us to give a
 compact, conceptual formula for these numbers. Let
 $\mathrm{n}_{1,1}(p)$ denote the number of immediate descendants of
 $\gp{1,1}{\Fp}$ of order $p^{10}$ and exponent~$p$. Set, moreover,
 $e(p)=|E_1[3](\F_p)|$ and $m(p)=\gcd(p-1,4)$.
     
  \begin{corollary}[du Sautoy--Vaughan-Lee]\label{cor:descendants}
    Assume that $p>3$. Then the following holds:
\[
\mathrm{n}_{1,1}(p)=\frac{p^2+p+2-m(p)+e(p)(p-5)+5m(p)e(p)}{m(p)e(p)}.
\]
\end{corollary}
\noindent
We hope to come back to the interesting question of how to generalize and
conceptualize this work to the groups $\gp{i,\para}{\ff}$ for $\para\in\Z\setminus\graffe{0}$ and $i\in\{1,2,3\}$
in a future paper.

\subsubsection{Further examples and related work}

In \cite{Lee/16}, Lee constructed an $8$-dimensional group scheme ${\bf G}$,
via a presentation akin to~\eqref{def:pres.G11}, and proved a result which is
similar to Theorem~\ref{th:main.dSVL}. In particular, he showed that both the
orders of $\Aut({\bf G}(\F_p))$ and the numbers of immediate descendants of
${\bf G}(\F_p)$ of order $p^9$ and exponent $p$ vary with $p$ in a
nonquasipolynomial way. More precisely, these numbers depend on the splitting
behaviour of the polynomial $x^3-2$ over $\F_p$ or, equivalently, on the
realisability over $\F_p$ of permutations of $3$ specific, globally defined
points in~$\mathbb{P}^1$.

In \cite{VL/18}, Vaughan-Lee proved a similar result about $p$-groups arising
from a parametrized family of $7$-dimensional unipotent group schemes of
nilpotency class~$3$. He proved that the orders of the automorphism groups of
these $p$-groups, which feature two integral parameters $x$ and $y$, depend on
the splitting behaviour modulo $p$ of the polynomial $t^3-xt-y\in\Z[t]$.

It is natural to try and further expand the range of computations of
automorphisms of $p$-groups obtained as $\ff$-points of
finite-dimensional unipotent group schemes, where $\ff$ is a finite
field.

Arithmetic properties of finite $p$-groups arising as groups of
$\ff$-rational points of group schemes defined in terms of matrices of
linear forms are also a common theme of \cite{BostonIsaacs04,
  O'BrienVoll/15, Rossmann/19, RossmannVoll/19} (with a view towards
the enumeration of conjugacy classes) and \cite{BMS/17} (with a view
towards faithful dimensions; cf.\ also Section~\ref{subsubsec:POFS}).

Slight variations of the group schemes ${\bf G}_{1,\para}$ were highlighted by
du Sautoy in the study of the (normal) subgroup growth of the groups of
$\Z$-rational points of these group schemes. Indeed, he showed in \cite{dS02}
that the local normal zeta functions of these finitely generated torsion-free
nilpotent groups depend essentially on the numbers of $\Fp$-rational points of
the reductions of the curves $E_{\para}$ modulo $p$; see also \cite{Voll/04,
  Voll/05} for explicit formulae and generalizations.

The groups we consider arise via a general construction of nilpotent
groups and Lie algebras from symmetric forms. Automorphism groups of
such groups and algebras have been studied, e.g.\ in
\cite{GrundhoeferStroppel/08, Wilson/17, BMW/17}. In fact, our
Theorem~\ref{th:main.para} may be seen as an explicit version of
\cite[Thm.~7.2]{GrundhoeferStroppel/08} and
\cite[Thm.~9.4]{Wilson/17}. The arithmetic point of view taken in the
current paper seems to be new, however.

\subsection{Organization and notation}
The paper is structured as follows. In Section~\ref{sec:groups.from.forms} we
recall the well-known general construction yielding the nilpotent groups and
Lie algebras considered in this paper and set up the notation that will be
used throughout the paper.  In Section~\ref{sec:curves} we gather some
elementary results about automorphisms of elliptic curves.  In
Section~\ref{sec:3} we collect a number of structural results of the groups in
question necessary to determine their automorphism groups.  We apply these, in
combination with the results of the previous sections, to give a proof of the
paper's main results and their corollaries in Section~\ref{sec:pfs}.

The notation we use is mainly standard. Throughout, $k$ denotes a
number field, with ring of integers~$\mcO_k$. The localization of
$\mcO_k$ with respect to the powers of $\para\in\Z$ is denoted
by~$R = \mcO_{k,\para}$. By $\alg$ we denote a field with an
$R$-algebra structure, in practice nearly always an extension of $k$
or a residue field of a nonzero prime ideal of $\mcO_k$. By $\ff$ we
denote a finite field.
  
\section{Groups and Lie algebras from (symmetric)
  forms}\label{sec:groups.from.forms}

In this section we work with groups and Lie algebras arising from symmetric
matrices of linear forms via a classical construction which we review in
Sections~\ref{subsec:setup} and~\ref{subsec:groups}. Our reasons to restrict
our attention to symmetric matrices, rather than discuss more general
settings, will become apparent in Section~\ref{sec:vvv}; see also
Remark~\ref{rmk:symmetry}. %\footnote{MS: added this part, CV rephrased it.}

\subsection{Global setup}\label{subsec:setup}
The following is the setup for the whole paper. Let $k$ be a number field with
ring of integers~$\mcO_k$. For a nonzero integer $\para$, we write
$\roc = \mcO_{k,\para}$ for the localization of $\mcO_k$ at the set of
$\para$-powers. Let further $d$ be a positive integer and let $U$, $W$, and
$T$ be free $\roc$-modules of rank~$d$. Let $\bfy = (y_1,\dots,y_d)$ be a
vector of independent variables and let $\tuB$ be a \ul{symmetric} matrix
whose entries are linear homogeneous polynomials (which may be $0$)
over~$\roc$, i.e.\
\begin{equation}\label{def:B}
  \tuB(\bfy) = \sum_{\kappa=1}^d \tuB^{(\kappa)}y_\kappa \in\Mat_d(\roc[y_1,\dots,y_d])
\end{equation}
for symmetric matrices $\tuB^{(\kappa)}\in\Mat_d(\roc)$. Let further
$\cor{E}=(e_1,\dots,e_d)$, $\cor{F}=(f_1,\dots,f_d)$, and
$\cor{T}=(g_1,\dots,g_d)$ be $\roc$-bases of $U$, $W$, and $T$
respectively, allowing us to identify each $U$, $W$, and $T$
with~$\roc^d$.  Let, moreover,
\begin{equation}\label{def:phi}
    \phi:U\times W\longrightarrow T, \quad (u,w) \longmapsto u \tuB w^{\tp} \stackrel{\textup{def}}{=} u \tuB(g_1,\dots,g_d) w^{\tp}
\end{equation}
be the $\roc$-bilinear map induced by $\tuB$ with respect to the given
bases.  Throughout we write $\otimes$ to denote~$\otimes_{\roc}$.  We
denote by
\begin{equation}\label{def:phi.tilde}
  \tilde{\phi}:U\otimes W\longrightarrow T
\end{equation}the homomorphism that is given by the
universal property of tensor products.  By slight abuse of notation, we will
use the bar notation ($\overline{\phantom{x}}$) both for the map
\begin{equation*}\label{eq: bar map}
  U\longrightarrow W, \quad u=\sum_{j=1}^du_je_j\longmapsto
  \overline{u}=\sum_{j=1}^du_jf_j
\end{equation*}
and its inverse $W\rightarrow U$.  We remark that the symmetry of
$\tuB$ yields, for $u\in U$ and $w\in W$,
\begin{equation}\label{eq:bars}
\phi(u,w)=u\tuB w^{\tp}=(u\tuB w^{\tp})^{\tp}=\overline{w}\tuB\overline{u}^{\tp}=\phi(\overline{w},\overline{u}).
\end{equation}
We write $$V = U \oplus W \textup{ and }\L=V \oplus T$$ for the free
$\roc$-modules of ranks $2d$ respectively~$3d$. Our notation will
reflect that we consider the summands of these direct sums as subsets.
\vspace{5pt}\\
\noindent
Throughout, $\alg$ is a field with an $\roc$-algebra
structure. (In practice, $\alg$ may be, for instance, an extension of the
number field $k$ or one of various residue fields of nonzero prime ideals of
$\mcO_k$ coprime to $\para$.) By slight abuse of notation we use, in the
sequel, the above $\roc$-linear notation also for the corresponding
$\alg$-linear objects obtained from taking tensor products over~$\roc$. For
example, $\L$ will also denote the $\alg$-vector space $\L\otimes_\roc \alg$.

\subsection{Groups and Lie algebras}\label{subsec:groups}
The data $(\alg,\tuB,\phi)$ gives the $\alg$-vector space $\L$ a group
structure. Indeed, with
\begin{align*}\label{def:star}
  \star: \L \times \L & \rightarrow \L\nonumber\\
  ((u,w,t),(u',w',t')) & \mapsto (u,w,t)\star(u',w',t')=(u+u',w+w',t+t'+\phi(u,w')),
\end{align*}
the pair $\gp{\tuB}{\alg} = (\L,\star)$ is a nilpotent group of class
at most~$2$, viz.\ the group of $\alg$-rational points of an
$\roc$-defined group scheme ${\bf G}_{\tuB}$ (abelian if and only if
$\tuB=0$.) It is not difficult to show that in $\gp{\tuB}{\alg}$
\begin{enumerate}[label=$(\arabic*)$]
\item the identity element is $(0,0,0)$,
\item the inverse of the element $(u,w,t)$ is $(u,w,t)^{-1}=(-u,-w,-t+\phi(u,w))$,
\item the commutator of any two elements $(u,w,t)$ and $(u',w',t')$
  is \begin{equation}\label{def:grp.com}
    [(u,w,t),(u',w',t')]=\phi(u,w')-\phi(u',w).
  \end{equation}
\end{enumerate}

\noindent
The data $(K,\tuB,\phi)$ also endows $\L$ with the structure of a graded
$\alg$-Lie algebra. Indeed, with
\begin{align}\label{def:lie.bra}
  [\,,]: \L\times \L&\rightarrow \L\nonumber\\ ((u,w,t),(u',w',t'))&\mapsto [(u,w,t),(u',w',t')]=\phi(u,w')-\phi(u',w),
\end{align}
the pair $\la{\tuB}{\alg} = (\L,[\,,])$ is a nilpotent $\alg$-Lie algebra of
nilpotency class at most two, viz.\ the $\alg$-rational points of an
$\roc$-defined Lie algebra scheme $\mathfrak{g}_\tuB$ (abelian if and only if
$\tuB=0$). Note that the Lie bracket~\eqref{def:lie.bra} coincides with the
group commutator~\eqref{def:grp.com}.

\begin{remark}\label{rmk:symmetry}
  The property of $\tuB$ being symmetric is not an isomorphism
  invariant of~$\gp{\tuB}{\alg}$. Indeed, linear changes of
  coordinates on $U$ and $W$ preserve the isomorphism type of
  $\gp{\tuB}{\alg}$ but result in a transformation $\tuB\mapsto
  P^{\tp}\tuB Q$ for some $P,Q\in\GL_d(\alg)$. For computations,
  however, the symmetric setting proved much more conventient to work
  with. Clearly, symmetric coordinate changes (i.e.\ $P=Q$) preserve
  the matrix's symmetry.

  We focus on groups $\gp{\tuB}{\alg}$ where $\tuB$ is a symmetric $3\times 3$
  determinantal representation of a planar elliptic curve $E$. As discussed in
  Section \ref{sec:gen.hess}, there are, over the algebraic closure of~$\alg$,
  three inequivalent such representations $\tuB_1$, $\tuB_2$, $\tuB_3$, where
  \emph{inequivalent} means that they belong to three distinct orbits under
  the standard action of $\GL_3(\alg)^2$ on $3\times 3$ matrices of linear
  forms. In particular, equivalent representations yield isomorphic groups,
  but, as our Theorem~\ref{th:iso} shows, inequivalent representations might
  yield isomorphic groups.
\end{remark}

\noindent
In the case that $\alg=\ff$ is a finite field of odd characteristic~$p$, the
groups $\gp{\tuB}{\ff}$ are the finite $p$-groups associated with the Lie
algebras $\la{\tuB}{\ff}$ by means of the classical Baer correspondence
(\cite{Baer/38}). It was anticipated, in the case at hand, in Brahana's work
(\cite{Brahana/35}) and extended in the much more general (and better known)
Lazard correspondence; \cite[Exa.~10.24]{Khukhro/98}. It implies, in
particular, that $|\Aut(\gp{\tuB}{\Fp})| = |\Aut(\la{\tuB}{\Fp})|$ and, more
generally, that $|\Aut(\gp{\tuB}{\ff})| = |\Aut_{\F_p}(\la{\tuB}{\ff})|$,
where $\Aut_{\F_p}(\la{\tuB}{\ff})$ denotes the automorphisms of
$\la{\tuB}{\ff}$ as $\Fp$-Lie algebra.

To lighten notation we will, in the sequel, not always notationally
distinguish between the $\alg$-Lie algebra scheme $\LA=\LA_\tuB$ and
its $\alg$-rational points. We trust that the respective contexts will
prevent misunderstandings.

\subsection{Groups of Lie algebra automorphisms}\label{sec:vvv}
By $\Aut_K(\LA)=\Aut(\LA)$ we denote the automorphism group of the
$K$-Lie algebra~$\LA$. Define, additionally,
\begin{align*}
\Aut_{V}(\LA) &= \graffe{\alpha\in\Aut(\LA) \mid \alpha(V)=V},\\
\Aut_V^{\mathrm{f}}(\LA) &= \graffe{\alpha\in\Aut_V(\LA) \mid \alpha(U)=U,\ \alpha(W)=W},\\
  \Aut_V^{=}(\LA) &= \graffe{\alpha\in\Aut^{\mathrm{f}}_V(\LA) \mid \forall u\in U: \alpha(\overline{u})=\overline{\alpha(u)}}.
\end{align*}
We remark that $\Aut_V(\LA)$ is nothing but the centralizer
$\Cyc_{\Aut(\LA)}(V)$ considered with respect to the action of
$\Aut(\LA)$ on the Grassmannian of~$\L$. Our choice of bases allows us
to identify $\Aut(\LA)$ with a subgroup of $\GL_{3d}(\alg)$. We may
thus view each element $\alpha$ of $\Aut(\LA)$ as a matrix of the form
\begin{equation}\label{eq:matricesshapes}
\alpha=
\left(
\begin{array}{ccc}
  A_U & A_{WU} & 0 \\
%  \hline
  A_{UW} & A_W & 0 \\
%  \hline 
  C & D & A_T
\end{array}
\right) \in \GL_{3d}(\alg),
\end{equation}
where $A_U,A_{WU},A_{UW},A_{W},C,D,A_T$ are elements of $\Mat_d(\alg)$. The
subgroup $\Aut_V(\LA)$ of $\Aut(\LA)$ comprises those matrices with entries
$C=D=0$ in \eqref{eq:matricesshapes}. It is easy to show that
\begin{equation}\label{eq:AutL}
\Aut(\LA) = \Hom_K(V,T)\rtimes\Aut_V(\LA) \cong K^{2d^2}\rtimes\Aut_V(\LA). 
\end{equation}
We also observe that every element of $\Aut_V^{\mathrm{f}}(\LA)$ is of
the form
%\begin{equation*}\label{eq:fixedUV}
%\diag(A_U,A_W,A_T)=
%\left[
%\begin{array}{c|c|c}
%A_U & 0 & 0 \\
%\hline
%0 & A_W & 0 \\
%\hline 
%0 & 0 & A_T
%\end{array}
%\right]
%\end{equation*}
\begin{equation*}\label{eq:fixedUV}
\diag(A_U,A_W,A_T)=
\left(
\begin{array}{ccc}
A_U & 0 & 0 \\
%\hline
0 & A_W & 0 \\
%\hline 
0 & 0 & A_T
\end{array}
\right)
\end{equation*}
and belongs to $\Aut_V^=(\LA)$ if and only if, in addition,~$A_U=A_W$.
\vspace{5pt}\\
\noindent
We define the map $\psi:\GL_2(\alg)\rightarrow\Aut(\LA)$ by
\begin{equation*}\label{def:psi}
M=\begin{pmatrix}
a & b \\
c & d
\end{pmatrix}\ \mapsto \ \psi(M)=\left(
  (u,w,t)\mapsto(au+b\overline{w},c\overline{u}+dw,(ad-bc)t) \right).
\end{equation*}

\begin{lemma}\label{lemma:GL2}
  The map $\psi$ is an injective homomorphism of groups.
\end{lemma}

\begin{proof}
  We show that $\psi$ is well-defined. Addition in the Lie algebra is clearly
  respected, so it suffices to show that the Lie brackets are respected,
  too. Let $(u,w,t)$ and $(u',w',t')$ in $\LA$ and let
  $M=\begin{pmatrix} a & b \\ c & d
\end{pmatrix}\in\GL_2(\alg)$. Set $\Delta=\det(M)$. Equation \eqref{eq:bars} then implies the
identities
\begin{align*}
[\psi(M)(u,w,t), \psi(M)(u',w',t')] & 
= [(au+b\overline{w},c\overline{u}+dw,\Delta t),(au'+b\overline{w'},c\overline{u'}+dw',\Delta t')] \\
& = \phi(au+b\overline{w},c\overline{u'}+dw')-\phi(au'+b\overline{w'},c\overline{u}+dw)\\
& = ad\phi(u,w')+bc\phi(\overline{w},\overline{u'})-ad\phi(u',w)-bc\phi(\overline{w'},\overline{u})\\
& = \Delta(\phi(u,w')-\phi(u',w)) \\ & 
=\psi(M)([(u,w,t),(u',w',t')]).
\end{align*}
To show that $\psi$ is an injective homomorphism is a routine check.  
\end{proof}
\noindent Let $u,u'\in U$ and let $w\in W$.  Then one easily computes
\begin{align*}\label{eq:UWabelian}
  [u,u'+w] & = \phi(u,w)-\phi(u',0)  = \phi(u,w) = \phi(\overline{w},\overline{u}) = \phi(\overline{w},\overline{u})-\phi(0,\overline{u'}) \nonumber\\
           &=
             [\overline{u'}+\overline{w},\overline{u}]=[\overline{u'+w},\overline{u}].
\end{align*}
In particular, $U$ is an abelian subalgebra of $\LA$ and, by symmetry,
so is $W$.  As a consequence of this fact and Lemma~\ref{lemma:GL2},
for each $M\in\GL_2(\alg)$, the subspaces $\psi(M)(U)$ and
$\psi(M)(W) =
\psi\left(M\left(\begin{matrix}0&1\\1&0\end{matrix}\right)\right)(U)$
are $3$-dimensional abelian subalgebras of $\LA$ that are contained
in~$V$.

\begin{lemma}\label{lemma:xreduction}
  Let $A_U$ and $A_W$ be elements of $\GL_d(\alg)$ and set
  $D=A_UA_W^{-1}$. Then the conditions
\begin{enumerate}[label=$(\arabic*)$]
\item there exists $A_T\in\GL_d(\alg)$ such that
  $\diag(A_U,A_W,A_T)\in \Aut_V^\mathrm{f}(\LA)$;
 \item the equality $D^{\tp}\tuB = \tuB D$ holds;
 \item the subspace $X=\graffe{w+\overline{w}D^{\tp} \mid w\in W}$ is a complement of $U$ in $V$ satisfying $[X,X]=0$;
\end{enumerate}
are related in the following way:
\[
(1) \ \ \Longrightarrow \ \ (2) \ \ \Longleftrightarrow \ \ (3).
\]
\end{lemma}

\begin{proof}
  $(1)\Rightarrow(2)$: A given element $\gamma\in\GL(T)$ corresponds,
  via the identification $T \cong \alg^d$ induced by the basis $\cor{T}$, to a
  matrix $A_T\in\GL_d(\alg)$. We denote by $\gamma(y_1),\ldots,\gamma(y_d)\in \alg[y_1,\ldots,y_d]$ the images of the variables $y_1,\ldots, y_d$ under $A_T$, acting on $\alg[y_1,\ldots,y_d]$ in the standard way.
  Writing $\tuB = \tuB(y_1,\ldots,y_d)$, we analogously define
$$\tuB_{\gamma} = \tuB(\gamma(y_1),\ldots,\gamma(y_d)) \in\Mat_d(\alg[y_1,\ldots,y_d]).$$
Given $A_U,A_W\in\GL_d(\alg)$, the matrix $\diag(A_U,A_W,A_T)$ is an element
of $\Aut_V^\mathrm{f}(\LA)$ if and only if there exists a $\gamma\in\GL(T)$
with $A_U^{\tp} \tuB A_W=\tuB_{\gamma}$. As $\tuB$ is symmetric, so is
$\tuB_{\gamma}$.  Hence
\[
A_U^{\tp} \tuB A_W=\tuB_{\gamma} \Longleftrightarrow A_W^{\tp} \tuB
A_U=\tuB_{\gamma}
\]
and thus $\diag(A_U,A_W,A_T)$ belongs to $\Aut_V^\mathrm{f}(\LA)$ if
and only if so does $ \diag(A_W,A_U,A_T)$.  Consequently,
$\diag(A_U,A_W,A_T)$ belongs to $\Aut_V^\mathrm{f}(\LA)$ if and only
if so is $\diag(A_W^{-1},A_U^{-1},A_T^{-1})$.  It follows that if the
matrix $\diag(A_U,A_W,A_T)$ belongs to $\Aut_V^\mathrm{f}(\LA)$, then
$D^{\tp}\tuB = \tuB D$ holds.

$(2)\Leftrightarrow (3)$: Clearly $X$ is a complement of $U$ in
$V$. We obtain that $X$ is an abelian subalgebra of $\LA$ if and only
if, for any choice of $w,w'\in W$, the element
$[w+\overline{w}D^{\tp},w'+\overline{w'}D^{\tp}]$ is trivial. This
happens if and only if, for all $w,w'\in W$, the following equalities
hold:
\[
\overline{w'}D^{\tp} \tuB w^{\tp}=\overline{w}D^{\tp}\tuB w'^{\tp}=(\overline{w}D^{\tp}\tuB w'^{\tp})^{\tp}=\overline{w'}\tuB D\overline{w}^{\tp}.
\]
As a consequence, $X$ is abelian if and only if $D^{\tp}\tuB=\tuB D$.
\end{proof}

\begin{lemma}\label{lemma:Prod+Intersection}
  Assume that all $d$-dimensional abelian subalgebras of $\LA$ that
  are contained in $V$ are of the form $\psi(M)(U)$ for some
  $M\in\GL_2(\alg)$.  Then the following hold:
  \begin{enumerate}[label=$(\arabic*)$]
  \item\label{first.new} $\Aut_V(\LA)=\psi(\GL_2(\alg))\Aut_V^{\mathrm{f}}(\LA)$.
  \item\label{second.new}
    $\alg^\times\rtimes\Aut_V^=(\LA) \cong \Aut_V^{\mathrm{f}}(\LA) \leq
    \graffe{\diag(\lambda A, A, A_T) \mid \lambda \in \alg^\times,
      A,A_T\in\GL_d(\alg)}$.
  \item\label{third.new}
    $\psi(\GL_2(\alg))\cap\Aut_V^{\mathrm{f}}(\LA)=\graffe{\diag(\lambda\Id_d,\nu\Id_d,\lambda\nu\Id_d)
      \mid \lambda,\nu\in K^\times} \cong K^\times \times K^\times$.
  \item\label{fourth.new} If $\alg=\ff$ is a finite field, then
$$
  |\Aut_V(\LA)| = \frac{|\GL_2(\ff)|\cdot
    |\Aut_V^{\mathrm{f}}(\LA)|}{(|\ff|-1)^2} = \frac{|\GL_2(\ff)|\cdot
    |\Aut_V^{=}(\LA)|}{|\ff|-1}.
$$
 
\end{enumerate}
\end{lemma}

\begin{proof}
  \ref{first.new}: Let $\alpha\in\Aut_V(\LA)$. By assumption, there
  is $M=\begin{pmatrix}
    a & b \\
    c & d
    \end{pmatrix}
    $ in $\GL_2(\alg)$ such that % and let $a,b,c,d\in \alg$ be such that
  \[
    \alpha(U)=\gen{ae_i+cf_i \mid i=1,\ldots,d} \textup{ and }
    \alpha(W)=\gen{be_i+df_i \mid i=1,\ldots,d}.
  \]
  For such an $M$, the element $\psi(M)^{-1}\circ\alpha$ belongs
  to~$\Aut_V^{\mathrm{f}}(\LA)$, equivalently $\alpha$ is an element
  of~$\psi(\GL_2(\alg))\Aut_V^{\mathrm{f}}(\LA)$. This shows that
  $\Aut_V(\LA)\subseteq \psi(\GL_2(\alg))\Aut_V^{\mathrm{f}}(\LA)$;
  the converse inclusion is clear.

  \ref{second.new}: Let $X$ be a $d$-dimensional abelian subalgebra of $\LA$
  that is contained in~$V$. It follows from the assumptions that either $X=U$
  or there exists $\lambda\in \alg$ such that
\[
X=\graffe{w+\lambda\overline{w} \mid w\in W}.
\] 
In other words, the complements of $U$ in $V$ that are also abelian
subalgebras of $\LA$ are parametrized by the scalar matrices in
$\Mat_d(\alg)$. Lemma~\ref{lemma:xreduction} yields that, if
$\diag(A_U,A_W,A_T)$ is an element of $\Aut_V^{\mathrm{f}}(\LA)$, then
there exists $\lambda\in \alg^\times$ such that $A_U=\lambda A_W$. By
Lemma~\ref{lemma:GL2}, the subgroup
\[
N=\graffe{\diag(\lambda\Id_d,\Id_d,\lambda\Id_d)}\cong \alg^{\times}
\]
of $\GL_{3d}(\alg)$ is actually a subgroup of $\Aut_V^{\mathrm{f}}(\LA)$,
which is easily seen to be normal. Moreover, the intersection
$N\cap\Aut_V^=(\LA)$ is trivial, by definition of $\Aut_V^=(\LA)$.

\ref{third.new}: Follows directly from the definition of
$\Aut_V^{\mathrm{f}}(\LA)$.

\ref{fourth.new}: Follows from combining~\ref{first.new}, \ref{second.new},
and~\ref{third.new}.
\end{proof}

\noindent In order to prove Theorem~\ref{th:main.para} in
Section~\ref{subsec:auto}, we will make use of
Lemma~\ref{lemma:Prod+Intersection} via~Proposition~\ref{prop:abel3d}.

Recall that, for any $A\in\GL_d(\alg)$, the isomorphism
$U\times W\rightarrow U\times W$, defined by
$(u,w) \mapsto (u A^{\tp}, w A^{\tp})$, induces a unique isomorphism
$A\otimes A:U\otimes W\rightarrow U\otimes W$. Recall, moreover, the
maps $\phi$ and $\tilde{\phi}$ defined in \eqref{def:phi} and
\eqref{def:phi.tilde}, respectively.

\begin{lemma}\label{lemma:kernel&tensors}
  Let $A\in\GL_d(\alg)$ and assume that the image of $\phi$
  spans~$T$. Then the following are equivalent:
  \begin{enumerate}[label=$(\arabic*)$]
  \item $(A\otimes A)(\ker\tilde{\phi})\subseteq\ker\tilde{\phi}$;
  \item there exists $A_T\in\GL_d(\alg)$ such that $\diag(A,A,A_T)\in\Aut(\LA)$.
  \end{enumerate}
\end{lemma}

\begin{proof}
  $(1)\Rightarrow(2)$: Assume that
  $(A\otimes A)(\ker\tilde{\phi})\subseteq\ker\tilde{\phi}$. Then
  $A\otimes A$ induces an isomorphism of $T$ in the following way. For
  $j\in\{1,\ldots,d\}$, choose an element $v_{j}\in U\otimes W$ such that
  $\tilde{\phi}(v_{j})=g_j$. Define $A_T:T\rightarrow T$ to be the
  $\alg$-linear homomorphism that is induced by
  $$g_j \mapsto A_T(g_j) = \tilde{\phi}((A \otimes A)v_j)$$
  for $j=1,\ldots,d$.  The map $A_T$ is well-defined since $A\otimes A$
  stabilizes the kernel of $\tilde{\phi}$. Moreover, the following
  diagram is commutative.
\[
\xymatrix{
U\otimes W \ar[r]^{\tilde{\phi}} \ar[d]_{A\otimes A} & T \ar[d]^{A_T} \\                  
U\otimes W \ar[r]^{\tilde{\phi}} & T
}
\]
Since $\tilde{\phi}\circ (A\otimes A)$ is surjective,  $A_T$ is
surjective and thus an isomorphism. In particular, 
$\diag(A,A,A_T):U\oplus W\oplus T\rightarrow U\oplus W\oplus T$ is an
automorphism of the $\alg$-Lie algebra $\LA$.

 $(2)\Rightarrow (1)$: Let
$u\in U$ and let $w\in W$, which we also regard as elements of
$\LA$. Then
\begin{align*}
A_T\circ\tilde{\phi}(u\otimes w) &  = A_T(\tilde{\phi}(u\otimes w)) = A_T([u,w]) \\
& = [uA^{\tp}, wA^{\tp}] = \phi(uA^{\tp},wA^{\tp})=\tilde{\phi}\circ(A\otimes A)(u\otimes w)
\end{align*}
holds and so we are done.
\end{proof}

\begin{remark}\label{rk:dual}
  We exploit the symmetry of $\tuB$, for example, in the following
  way.  Let $v=\sum_{i,j=1}^da_{ij}e_i\otimes f_j$ be an element of
  $U\otimes W$ and set $C=A^{\tp}\tuB A$, which is by definition a
  symmetric matrix. Then we have the identities
  \begin{align*}
    \tilde{\phi}((A\otimes A)(v)) & = \tilde{\phi}\left(\sum_{i,j=1}^da_{ij}(e_iA^{\tp})\otimes(f_jA^{\tp})\right) =\sum_{i,j=1}^da_{ij}\tilde{\phi}\left((e_iA^{\tp})\otimes(f_jA^{\tp})\right) \\
                                  & = \sum_{i,j=1}^da_{ij}\phi\left((e_iA^{\tp}),(f_jA^{\tp})\right)  = \sum_{i,j=1}^da_{ij}e_iA^{\tp}\tuB Af_j^{\tp} = \sum_{i,j=1}^da_{ij}e_iCf_j^{\tp}.
\end{align*}
After defining the ``dual'' element $v^*$ of $v$ as
$v^*=\sum_{i,j=1}^da_{ij}e_j\otimes f_i$, they entail the identity
$\tilde{\phi}((A\otimes A)(v))=\tilde{\phi}((A\otimes A)(v^*))$. We
will put this to use in Section~\ref{subsec:auto}.
\end{remark}

\subsection{Alternative descriptions}
The groups $\gp{\tuB}{\alg}$ defined in Section~\ref{subsec:groups}
have alternative descriptions as follows.

\subsubsection{Heisenberg groups}\label{subsubsec:heisenberg}
Let $\bfH$ be the group scheme of upper unitriangular
$3\times 3$-matrices. Then $\gp{\tuB}{\alg}$ is equal to~$\bfH(\mcA)$,
where $\mcA = \alg g_1\oplus \dots \oplus \alg g_d$ is the
$\alg$-algebra given by setting
  $$g_r \cdot g_s = \left( \tuB(g_1,\dots,g_d)\right)_{rs}$$
  for $r,s\in\{1,\dots,d\}$. Note that the algebra~$\mcA$ is
  commutative (as the matrix $\tuB$ is symmetric) but in general
  not associative. The group $\bfH(\mcA)$ is nilpotent of class at
  most $2$, and abelian if and only if $\tuB=0$. If $\mcA$ is associative, then $\bfH(\mcA)$ is called a
  \emph{Verardi group}; c.f.\ \cite{GrundhoeferStroppel/08}.

  We may recover the Lie algebra $\la{\tuB}{\alg}$ defined in
  Section~\ref{subsec:groups} from the group~$\gp{\tuB}{\alg}$. Indeed, we
  find that $\la{\tuB}{\alg}$ is isomorphic to the graded Lie algebra
  $$\left(\gp{\tuB}{\alg}/Z(\gp{\tuB}{\alg})\right) \oplus Z(\gp{\tuB}{\alg}),$$
  with Lie bracket induced by the commutator in $\gp{\tuB}{K}$. It is
  easy to show that this isomorphism induces a dimension-preserving
  bijection between subalgebras of $\la{\tuB}{\alg}$ and subgroups
  of~$\gp{\tuB}{\alg}$; cf.\ also \cite[Rem.~on p.~206]{GSS/88}. In
  view of the identification of the group $\gp{\tuB}{\alg}$ with
  $\bfH(\mcA)$ we may identify $\la{\tuB}{\alg}$ with $\mfh(\mcA)$,
  where $\mfh$ is the Lie algebra scheme of strict uppertriangular
  $3\times 3$-matrices.

\subsubsection{Alternating hulls of (symmetric) module
  representations}\label{subsubsec:alt.hull}
Specifying $(\alg,\tuB,\phi)$ amounts to defining a module representation
$\theta^\bullet:T^* \rightarrow \Hom(W,U^*)$ in the sense of
\cite[\S~2.2]{Rossmann/19}, the ``bullet dual'' of a module representation
$\theta: U \rightarrow \Hom(W,T)$ as explained in
\cite[\S~4.1]{Rossmann/19}. Here, for an $\roc$-module $M$ we write $M^*$ for
the usual dual. The group $\gp{\tuB}{\alg}$ is isomorphic to the group
$\mathsf{G}_{\Lambda(\theta^\bullet)}(\alg)$ associated with the alternating
hull $\Lambda(\theta^\bullet)$ of the bullet dual $\theta^\bullet$ of $\theta$
(see \cite[\S~7.2]{Rossmann/19}) in \cite[\S~7.1]{Rossmann/19} and equal to
the group $\mathsf{H}_{\theta^\bullet}(\alg)$ defined
in~\cite[\S~7.3]{Rossmann/19}. As remarked by Rossmann, this construction also
features in~\cite[\S~9.2]{Wilson/17}.

\subsubsection{Central extensions and cohomology}
The group $\gp{\tuB}{\alg}$ is a central extension of the abelian
group $V=U\oplus W$ by $T$ and corresponding to the $2$-cocycle
$$\phi^V:V\times V\rightarrow T, \quad (u+w,u'+w') \longmapsto 
\phi(u,w')$$
in $\mathrm{Z}^2(V,T)$, where the action of $V$ on $T$ is taken to be
trivial. This classical construction can be found, for example,
in~\cite[Ch.~IV]{Brown/82}. The {equivalence classes} of central
extensions of $V$ by $T$, in the sense of \cite[Ch.~IV.1]{Brown/82},
are in $1$-to-$1$ correspondence with the elements
of~$\hc^2(V,T)$. Equivalence being a stronger notion than isomorphism,
the image of $\phi^V$ in~$\hc^2(V,T)$ will generally not suffice to
determine the isomorphism type of~$\gp{\tuB}{\alg}$.

\section{Automorphisms and torsion points of elliptic
  curves}\label{sec:curves}
    
\subsection{General notation and standard facts}\label{sec:3.1}
  
\noindent
Let $\alg$ be a field and let $E$ be an elliptic curve over~$\alg$ with point at
infinity $\neec$. For a positive integer $n$, we write $E[n]$ for the
$n$-torsion subgroup of $E$ and $E(\alg)$ resp.\ $E[n](K)$ for the
$\alg$-rational points of~$E$ resp.\ $E[n]$. We define, additionally,
\[
\Aut(E)=\graffe{\varphi:E\rightarrow E \mid \textup{
    $\varphi$ automorphism of $E$ as projective curve}},
\]
containing as a subgroup the automorphisms of $E$ as elliptic
curve, or invertible
isogenies, 
\[
  \Aut_{\neec}(E)=\graffe{\varphi:E\rightarrow E \mid
    \varphi \in \Aut(E), \varphi(\neec) = \neec};
\]
for more
information, see for example \cite[Chap.~III.4]{silverman}. Except for the notation for the
automorphism groups, we will refer to results from and notation used
in~\cite[Chap.~III]{silverman}. We warn the reader that in
\cite{silverman} the notation $\Aut(E)$ denotes what we
defined as~$\Aut_{\neec}(E)$.

In Section~\ref{sec:lifting} we determine, for specific elliptic
curves, which of their endomorphisms are induced by (or lift to)
linear transformations of the plane. To this end we define 
%the subset $\XE\subseteq\GL_3(K)$ by its image
$$\overline{\XE}=\graffe{\overline{\varphi} \mid \varphi\in\GL_3(K), \ \overline{\varphi}(E)\subseteq E}\subseteq \PGL_3(K).$$
We write, additionally, $\cE$ for the natural
homomorphism
\[
\cE:\overline{\XE}\longrightarrow\Aut(E), \quad \overline{\varphi} \longmapsto \overline{\varphi}_{|E}.
\]

\begin{remark}\label{rk:autos}
  Any endomorphism $\varphi$ of $E$ as a projective curve can be
  written as a composition $\varphi = \tau\circ\alpha$, where $\tau$ is a
  translation and $\alpha$ is an isogeny $E\rightarrow E$.  In
  other words, each element of $\Aut(E)$ can be written as the
  composition of a translation with an element of
  $\Aut_{\neec}(E)$. There thus exists an isomorphism
\begin{equation*}\label{eq:aut.semi}
\Aut(E)\cong E\rtimes \Aut_{\neec}(E).
\end{equation*}
For more information see, for example, \cite[Exa.~III.4.7]{silverman}.
\end{remark}

\begin{lemma}\label{lemma:3outof4}
  Assume that $K$ is algebraically closed. Then there exists a subset
  $\cor{U}$ of $\alg^3$ of cardinality $4$ such that the projective image of
  $\cor{U}$ is contained in $E(\alg)$ and any $3$ elements of $\cor{U}$ form a
  basis of $K^3$.
\end{lemma}

\begin{proof}
  Let $Q_1,\ldots, Q_4\in E(\alg)$ be four points of coprime orders
  $|Q_i|$ with respect to addition in~$E$. Then, for each
  $(i,j)\in\graffe{1,2,3,4}^2$, the order $|Q_i+Q_j|$ is the least
  common multiple of $|Q_i|$ and $|Q_j|$. Hence the definition of the
  group law implies that no three points of $\graffe{Q_1,Q_2,Q_3,Q_4}$
  lie on the same projective line. Any collection of non-zero lifts of
  the $Q_i$ to elements of $\alg^3$ will do.
\end{proof}

\begin{lemma}\label{lemma:barcgeneral}
  The map $\cE:\overline{\XE}\longrightarrow\Aut(E)$ is injective.
\end{lemma}

\begin{proof}
  Without loss of generality we assume that $K$ is algebraically
  closed. Let $\varphi\in\GL_3(K)$ be such that
  $\overline{\varphi}\in\ker\cE$. Let, moreover, $\tilde{E}$ be the
  affine variety corresponding to $E$ in~$K^3$. Since $\varphi$
  induces the identity on $E$, every element of $\tilde{E}(K)$ is an
  eigenvector of $\varphi$. Then, Lemma~\ref{lemma:3outof4} yields not
  only that $\varphi$ is diagonalisable, but also that all
  eigenvectors have the same eigenvalue. In particular, there exists
  $\lambda\in K^\times$ such that $\varphi$ equals scalar
  multiplication by $\lambda$ and thus $\overline{\varphi}$ is
  trivial.
\end{proof}

\subsection{A parametrized family of elliptic
  curves}\label{sec:lifting}

\noindent
Let $\para\neq 0$ be an integer, let $\alg$ be a field of
characteristic not dividing~$2\para $, and let $E_{\para}$ be the
elliptic curve defined over~$\alg$ by
\begin{equation}\nonumber%\label{eq:Wf}
y^2=x^3-\frac{1}{\para}x.
\end{equation}
The projectivisation of $E_{\para}$, obtained by setting
$z = \para^{-1}$, is given by
\begin{equation}\label{eq:projcurvepara}
x^3 - \para xz^2 - \para y^2z = 0
\end{equation}
and has point at infinity equal to $\neec=(0:1:0)$. The $j$-invariant
of $E_{\para}$ being equal to $1728$, the automorphism group
$\Aut_{\neec}(E_{\para})$ consists of all maps of the form
$(x,y) \mapsto \ (\omega^2x+\rho, \omega^3y)$, where $(\omega,\rho)$
satisfies $\omega^4=1$ and $\para\rho^3=\rho$ with $\rho\neq 0$ only
if $\char(\alg)=3$; see \cite[Thm.~III.10.1,
Prop.~A.1.2]{silverman}. To lighten the notation, we will write
$\XEpara$ for $\cor{X}_{E_{\para}}$ and $\cpara$ for
$\overline{c}_{E_{\para}}$.

\begin{lemma}\label{lemma:isogenieslift.para}
  Let $(\omega,\rho)$ be such that $\omega^4=1$ and $\para\rho^3=\rho$
  with $\rho\neq 0$ only if $\char(\alg)=3$. Let further
  $\alpha\in\Aut_{\neec}(E_{\para})$ be defined by
  $({x},{y})\mapsto(\omega^2{x}+\rho,\omega^3{y}).$ Then the matrix
$$
\overline{A}=\begin{pmatrix}
  \omega^2 & 0 & \rho \\
  0 & \omega^3 & 0 \\
  0 & 0 & 1
 \end{pmatrix}
 \in \PGL_3(\alg)$$
 belongs to $\oXEpara$ and
 $\cpara^{-1}(\alpha)=\graffe{\overline{A}}$.
\end{lemma}

\begin{proof}
  Proving that $\overline{A}$ induces $\alpha$ is straightforward;
  uniqueness follows from Lemma~\ref{lemma:barcgeneral}.
\end{proof}

\noindent
In the next lemma, let $\sqrtpara$ be a solution of $X^2-\para=0$ and
set $\widetilde{K}=K[\sqrtpara]$.  Let, moreover,
$$P_1=(0,0),\ P_2=(\sqrtpara,0), \ \textup{and} \ P_3=(-\sqrtpara,0);$$ these are the nontrivial
$2$-torsion points of $E_{\para}$ over $\widetilde{K}$. For a point
$Q$ of $E_{\para}$ we denote, additionally, by $\tau_Q$ the
translation
$$\tau_Q:E_{\para}\rightarrow E_{\para},\quad P\mapsto P+Q.$$ As
pointed out in Remark~\ref{rk:autos}, the map $\tau_Q$ is an element
of $\Aut(E_{\para})$.

\begin{lemma}\label{lemma:QneqP1P2P3.para}
  Let $Q=(a,b)$ be a point of $E_{\para}$ such that
  $\tau_Q\in\im\overline{c}_{\para}$. Then
  $Q\notin\graffe{P_1,P_2,P_3}$.
\end{lemma}

\begin{proof}
  Let $A\in\GL_3(\alg)$ be such that
  $\overline{c}_{\para}(\overline{A})=\tau_Q$. %We will show that $Q\not\in\{P_1,P_2,P_3\}$, using addition formulas from \cite[Chap.~III.2]{silverman}.
  Without loss of generality, assume that $K=\widetilde{K}$.  For a
  contradiction we assume, moreover, that $Q=P_1$; the other cases are
  analogous. Since $P_1$ is an element of order $2$, we find that
\[
\tau_Q(\neec)=P_1, \ \ \tau_Q(P_1)=\neec, \ \ \tau_Q(P_2)=P_3, \ \ \text{and} \ \ \tau_Q(P_3)=P_2.
\]
As a consequence, there exist $\nu,\gamma,\varepsilon\in K^\times$
such that
\begin{align*}
(0,0,1)A^{\tp} = \nu(0,1,0), \quad (\sqrtpara,0,\para^{-1})A^{\tp}
 =\gamma(-\sqrtpara,0,\para^{-1}), \quad
(-\sqrtpara,0,\para^{-1})A^{\tp} & =\varepsilon(\sqrtpara,0,\para^{-1}).
\end{align*}
It follows that
$$
2\nu(0,1,0)=(0,0,2)A^{\tp}=\para((\sqrtpara,0,\para^{-1})+(-\sqrtpara,0,\para^{-1}))A^{\tp}=\gamma(-\para^{\frac{3}{2}},0,1)+\varepsilon(\para^{\frac{3}{2}},0,1)
$$
and therefore $\nu=0$. Contradiction.
\end{proof}

\begin{lemma}\label{lemma:necesLIFT.para}
Let $Q=(a,b)$ be a point of $E_{\para}$ such that $\tau_Q\in\im\overline{c}_{\para}$. Then $3\para^2a^4-6\para a^2-1=0$.
\end{lemma}

\begin{proof}
  Let $A=(a_{ij})_{i,j}\in\GL_3(\alg)$ be such that
  $\overline{c}_{\para}(\overline{A})=\tau_Q$. Observe that $\tau_Q$
  maps the point $\neec$ to $Q$ and its inverse $-Q$ to
  $\neec$. Moreover, if $(x,y)$ is an element of
  $E_{\para}\setminus\graffe{\pm Q}$, then
\begin{equation}\label{eq1S}
\tau_{Q}((x,y)) = %\mapsto
\bgr{
\frac{a_{11}x+a_{12}y+\para^{-1}a_{13}}{\para a_{31}x+\para a_{32}y+a_{33}},
\frac{a_{21}x+a_{22}y+\para^{-1}a_{23}}{\para a_{31}x+\para a_{32}y+a_{33}}
}\in E_{\para}.
\end{equation}
However, using the addition formulas (see e.g.\
\cite[Chap.~III.2]{silverman}) and assuming that $(x,y)\neq\pm Q$, one
also gets that
\begin{equation}\label{eq2S}
\tau_Q((x,y)) = (x',y')=\bgr{
\frac{(b-y)^2-(a+x)(a-x)^2}{(a-x)^2},
\frac{(y-b)x'-(ay-bx)}{a-x}
}.
\end{equation}
By Lemma~\ref{lemma:QneqP1P2P3.para} $Q\notin\graffe{P_1,P_2,P_3}$,
and thus \eqref{eq1S}, \eqref{eq2S}, and the fact that
$b^2=a^3-\para^{-1}a$ imply that
\begin{alignat*}{3}
  \tau_Q(P_1) &=\Bigg(\frac{\para^{-1}a_{13}}{a_{33}},\frac{\para^{-1}a_{23}}{a_{33}}\Bigg)& &=\bgr{-\frac{1}{\para a},\frac{b}{\para a^2}} \\
  \tau_Q(P_2) &=\bgr{\frac{\sqrtpara a_{11} +\para^{-1} a_{13}}{\para^{\frac{3}{2}} a_{31} + a_{33}},\frac{\sqrtpara a_{21}+\para^{-1} a_{23}}{\para^{\frac{3}{2}} a_{31}+ a_{33}}}& &=\bgr{\frac{\sqrtpara a+1}{\para a-\sqrtpara},-\frac{2b}{(\sqrtpara a-1)^2}}\\
  \tau_Q(P_3) &=\bgr{\frac{-\sqrtpara a_{11}+\para^{-1}
      a_{23}}{-\para^{\frac{3}{2}} a_{31}+ a_{33}},\frac{-\sqrtpara
      a_{21}+\para^{-1} a_{23}}{-\para^{\frac{3}{2}} a_{31}+ a_{33}}}&
  &=\bgr{\frac{1-\sqrtpara a}{\sqrtpara +\para
      a},-\frac{2b}{(\sqrtpara a+1)^2}}.
\end{alignat*}
As a consequence, we are looking to solve the following system of
equations:
\begin{itemize}
\item[(E1)]$-aa_{13}=a_{33}$, 
\item[(E2)]$a^2a_{23}=ba_{33}$,
\item[(E3)]$(\para a-\sqrtpara)(\sqrtpara a_{11}+\para^{-1} a_{13})=(\sqrtpara a+1)(\para^{\frac{3}{2}} a_{31}+ a_{33})$,
\item[(E4)]$(\sqrtpara a-1)^2(\sqrtpara a_{21}+\para^{-1} a_{23})=-2b(\para^{\frac{3}{2}} a_{31}+ a_{33})$,
\item[(E5)]$(\para a+\sqrtpara)(-\sqrtpara a_{11}+\para^{-1} a_{13})=(1-\sqrtpara a)(-\para^{\frac{3}{2}} a_{31}+ a_{33})$,
\item[(E6)]$(\sqrtpara a+1)^2(-\sqrtpara a_{21}+\para^{-1} a_{23})=-2b(-\para^{\frac{3}{2}} a_{31}+ a_{33})$.
\end{itemize}
Combining (E3) and (E4) one gets
\[
-2\sqrtpara b(\para^{\frac{3}{2}} a_{11}+ a_{13})=(\para a^2-1)(\para^{\frac{3}{2}} a_{21}+ a_{23})
\]
while combining (E5) and (E6) one gets
\[
-2\sqrtpara b(-\para^{\frac{3}{2}} a_{11}+ a_{13})=(1-\para a^2)(-\para^{\frac{3}{2}} a_{21}+ a_{23}).
\]
From the last two equalities, we derive
\begin{enumerate}
\item[(E7)]$-2 ba_{13}=(\para a^2-1)\para a_{21}$
\item[(E8)]$-2\para^2 ba_{11}=(\para a^2-1)a_{23}$
\end{enumerate}
Combining (E4) and (E6) one gets 
\begin{align*}
-2\para^{\frac{3}{2}} ba_{31} = (\sqrtpara a-1)^2(\sqrtpara a_{21}+\para^{-1} a_{23})+2b a_{33}  = -(\sqrtpara a+1)^2(-\sqrtpara a_{21}+\para^{-1} a_{23})-2b a_{33}
\end{align*}
from which it follows that
\[
2b a_{33}-2\para aa_{21}+(\para a^2+1)\para^{-1} a_{23}=0.
\]
With the use of (E2) for $a_{23}$ and of (E7) and (E1) for $a_{21}$, one may
rewrite this equation as
\begin{align*}
0 & = 2ba_{33}-2\para aa_{21}+(\para a^2+1)\para^{-1}a_{23} 
  = 2ba_{33}-\frac{4ba_{33}}{\para a^2-1}+\frac{(\para a^2+1)ba_{33}}{\para a^2} \\
  & =\frac{ba_{33}}{\para a^2(\para a^2-1)}(3\para^2a^4-6\para a^2-1).
\end{align*}
But $ba_{33}$ is nonzero, since $A$ is invertible and
$\pm Q\notin\graffe{P_1,P_2,P_3}$, and so it follows that
$3\para^2a^4-6\para a^2-1=0$.
\end{proof}

\noindent
We now give a geometric interpretation of the quartic polynomial
featuring in Lemma~\ref{lemma:necesLIFT.para}.

\begin{lemma}\label{lemma:3torEQ.para}
  Let $Q=(a,b)$ be an element of $E_{\para}$. Then the following are
  equivalent:
\begin{enumerate}[label=$(\arabic*)$]
\item $Q$ is a $3$-torsion point of $E_{\para}$;
\item $Q$ is a flex point of $E_{\para}$;
\item $3\para^2a^4-6\para a^2-1=0.$
\end{enumerate}
\end{lemma}

\begin{proof}
  $(1)\Leftrightarrow(2)$: This is classical; see, for instance,
  \cite[Exe.~5.37]{Fulton69}.  

  $(2)\Leftrightarrow(3)$: Let $f_{\para}$ be as in \eqref{def:fS}.
  The point $(a,b)\in E_{\para}$ is a flex point if and only if it is
  a solution to
  $$0=\Hes(f_{\para}(x,y,\para^{-1}))=8(1-3\para^2xy^2+3\para x^2).$$
  Using $b^2=a^3-\para^{-1}a$ one easily checks that this holds if and
  only if $3\para^2a^4-6\para a^2-1=0$.
\end{proof}

\begin{proposition}\label{prop:lifting IsomE.para}
The image of $\cpara$ is
%  The subgroup of $\Aut(E_{\para})$ of those automorphisms of $E_{\para}$ that  are induced by elements of $\PGL_3$ is 
  isomorphic to a subgroup of
  $E_{\para}[3]\rtimes\Aut_{\neec}(E_{\para}).$
\end{proposition}

\begin{proof}
Combine Remark~\ref{rk:autos} with Lemmas~\ref{lemma:isogenieslift.para},~\ref{lemma:necesLIFT.para}, and~\ref{lemma:3torEQ.para}.
\end{proof}

\begin{lemma}\label{lemma:oneil}
  Let $Q=(a,b)$ be a point of $E_{\para}[3]$ and assume that
  $A\in\GL_3(K)$ is such that $\cpara(\overline{A})=\tau_Q$.  Then
  there exists $\nu\in K^{\times}$ such that, up to a scalar,
\[
A=\begin{pmatrix}
\para ab+2ab\nu & \para a^2 & -2\para a^2b\nu \\
(-3\para-2\nu)b^2 & \para ab & 2\para ab^2\nu  \\
(1-2\nu a^2)b & a & 2\para a^3b\nu 
\end{pmatrix}.
\]
\end{lemma}

\begin{proof}
  From the addition formulas we derive
\begin{alignat*}{2}
\tau_Q(\neec)=Q&=(a:b:\para^{-1}), & \quad \tau_Q((0:0:1))&=(-(\para a)^{-1}:b(\para a^2)^{-1}:\para^{-1}),\\
\tau_Q(Q)=[2]Q=-Q&=(a:-b:\para^{-1}),
& \quad \tau_Q([2]Q)=\neec&=(0:1:0).
\end{alignat*}
Moving to affine coordinates $(u_1,u_2,u_3)$ we may find
$\lambda,\nu,\gamma\in K^\times$ such that
\begin{alignat*}{2}
 (0,1,0)A^{\tp} &=\lambda(\para a,\para b,1), & \quad
 (0,0,1)A^{\tp}  &=\nu (-a,b,a^2), \\
 (\para a,\para b,1)A^{\tp} &= (\para a,-\para b,1), & \quad
 (\para a,-\para b,1)A^{\tp} &= \gamma(0,1,0).
\end{alignat*}
%Fix such $\lambda,\nu,\gamma$ and, since 
Since $2\para b(0,1,0) = (\para a,\para b,1) - (\para a,-\para b,1)$, %, the linearity
%of $A$ implies then the following system
we deduce that
\begin{align*}
\para a = 2\para^2ab\lambda, \quad 
-\para b-\gamma = 2\para^2b^2\lambda, \quad
1 &= 2\para b\lambda,
\end{align*}
whence $\gamma=-2\para b$ and $\gamma\lambda=-1$.  Up to a scalar, we
may rewrite the above system as
\begin{alignat*}{2}
 (0,1,0)A^{\tp} &=(a(2b)^{-1},2^{-1},(2\para b)^{-1}), & \quad
 (0,0,1)A^{\tp}  &=\nu (-a,b,a^2), \\
 (\para a,\para b,1)A^{\tp} &= (\para a,-\para b,1), & \quad
 (\para a,-\para b,1)A^{\tp} &= (0,-2\para b,0).
\end{alignat*}
This implies that
$$
(1,0,0)A^{\tp}=\bgr{\frac{\para+2\nu}{2\para}, \frac{(-3\para-2\nu)b}{2\para
    a}, \frac{1-2\nu a^2}{2\para a} }.
$$
Multiplying $A$ by $2\para ab$ gives the claim.
\end{proof}

\noindent
We close this section with an observation in the special case when
$\para=\varepsilon^4$ for some $\varepsilon\in\alg$. As explained in
Section~\ref{subsec:intro.dSVL}, it is used to
establish~\eqref{eq:aut}.  To contextualize this situation, we remark
that, when $\para'\in\Z\setminus\{0\}$ and $6\para'$ is not divisible
by $\char(\alg)$, then the elliptic curves $E_\para $ and $E_{\para'}$
are isomorphic over $\alg$ if and only if there exists some
$\varepsilon\in \alg$ such that $\para=\para'\varepsilon^4$; cf.\
\cite[Ch.~III.1]{silverman}. Indeed, an isomorphism
$E_{\para}\rightarrow E_{\para'}$ is given by the invertible isogeny
$(x,y)\mapsto(\varepsilon^2 x, \varepsilon^3 y)$.

\begin{lemma}\label{lem:dSVL.torsion.points}
  Assume that $\char(\alg)\neq 3$ and that $\para=\varepsilon^4$ for
  some $\varepsilon\in K$. Define, moreover,
  \begin{align*}
   \mcS_0  &= \{(a,b)\in K^2 \mid b^2=a^3-a \textup{ and } a^4+6a^2 -3 = 0 \}, \\
   \mcS_1  &= \{(a,b) \in K^2 \mid b^2=a^3-a \textup{ and } 3a^4 - 6a^2 -1 = 0 \}, \\
   \mcS_{1,\para}  &= \{(a,b) \in K^2 \mid b^2=a^3-\para^{-1}a \textup{ and } 3\para^2a^4 - 6\para a^2 -1 = 0 \}.
\end{align*}
Then there exists bijections
\[
E_{\para}[3](\alg)\setminus\graffe{\neec} \longrightarrow
\mcS_{1,\para} \longrightarrow \mcS_1 \longrightarrow \mcS_0.
\]
\end{lemma}

\begin{proof}
  Lemma~\ref{lemma:3torEQ.para} yields a bijection
  $E_\para[3](\alg)\setminus\graffe{\neec} \rightarrow
  \mcS_{1,\para}$.
  A bijection $\mcS_{1,\para} \rightarrow \mcS_1$ is given
  by the restriction of the isogeny $(a,b) \mapsto (\varepsilon^2a, \varepsilon^3b)$.
  For the last arrow note that, with
  $$\mcS_2 = \{ (a,b) \in K^2 \mid 1-a^2 + ab^2 = 0 \textup{ and } a^4 + 6a^2 - 3 = 0 \}$$
  and the fact that $\char(\alg)\neq 3$, the maps
$$    \mcS_1 \longrightarrow \mcS_2, \, (a,b) \longmapsto \left(-\frac{1}{a},\frac{b}{a}\right), \ \textup{ and }\
\mcS_2 \longrightarrow \mcS_1, \, (a,b) \longmapsto
\left(-\frac{1}{a},-\frac{b}{a}\right),$$ are well-defined and mutually
inverse bijections. The sets $\mcS_2$ and $\mcS_0$ are in bijection as, for a
fixed element $a\in K$ with $a^3+6a^2-3=0$, we find that $a\neq 0$ and thus
the solutions to the equation $b^2=a^3-a$ are in bijection with the solutions
of the equation $b^2=\frac{a^2-1}{a}$.
\end{proof}

\section{Degeneracy loci and automorphisms of $p$-groups}\label{sec:3}
This section brings together the constructions from
Section~\ref{sec:groups.from.forms} and the facts about automorphisms
of elliptic curves from Section~\ref{sec:curves}. In the case that the
determinant of the matrix of linear forms $\tuB$ determines a
{elliptic} curve~$E$, we define an explicit homomorphism
\[
  \overline{c}_{\tuB}: \Aut_V^=(\LA_{\tuB})\rightarrow\Aut(E).
\]
In the case that, in addition, $\tuB$ is a Hessian matrix and $\ff$ is
a finite field of odd characteristic, the current section's main
result Corollary~\ref{cor:formula} yields a formula for
$|\Aut(\la{\tuB}{\ff})|$ in terms of the size of the image
of~$\overline{c}_{\tuB}$. To prove Theorem~\ref{th:main.para} in
Section~\ref{subsec:auto} we are just left with determining this image
size explicitly.

Throughout Section~\ref{sec:3} we continue to use the notation
introduced in Section~\ref{subsec:setup}. Recall, in particular, the
definition \eqref{def:B} of the matrix of linear forms
$\tuB({\bf y}) \in \Mat_d(\roc[y_1,\dots,y_d])$ in terms of
``structure constants'' $\tuB^{(\kappa)}_{ij}\in\roc$.

\subsection{Duality, degeneracy loci, and centralizer
  dimensions}\label{sec:spec}
\begin{definition}
  Let $\bfx=(x_1,\dots,x_d)$ be a vector of algebraically independent
  variables. The \emph{dual} of $\tuB(\bfy)$ is defined as
  \begin{equation*}
 \tuBd(\bfx) = \left( \sum_{j=1}^d
  \tuB^{(\kappa)}_{ij}x_j\right)_{i\kappa} \in\Mat_d(\roc[x_1,\dots,x_d]). %\\
\end{equation*}
\end{definition}

\begin{remark}\label{rem:duality}\
 % \begin{enumerate}
  The matrices $\tuB(\bfy)$ and $\tuBd(\bfx)$ are, in a precise sense, dual to
  one another. Indeed, as we pointed out in Section~\ref{subsubsec:alt.hull},
  the matrix $\tuB$ characterizes a module representation
  $\theta^\bullet:T^*\rightarrow\Hom(W,U^*)$. The dual matrix $\tuBd$ then
  characterizes the module representation
  $\theta=(\theta^\bullet)^\bullet:U\rightarrow \Hom(W,T)$; see
  \cite[\S~4.1]{Rossmann/19} and Definition~\ref{def:A}. In particular, they
  satisfy
\begin{equation}\label{eq:A-B}
\tuBd(\mathbf{x})\mathbf{y}^{\tp}=\tuB(\mathbf{y})\mathbf{x}^{\tp}.
\end{equation}
The matrices $\tuB(\bfy)$ and $\tuBd(\bfx)$ are also closely related to the
``commutator matrices" defined in \cite[Def.~2.1]{O'BrienVoll/15}. In this
paper's notation, we find that, for vectors of algebraically independent
variables $\bfY=(Y_1,\dots,Y_d)$ and
$\bfX=(\bfX_1,\bfX_2) = (X_1,\dots,X_d, X_{d+1}, \dots, X_{2d})$,
$$B(\bfY)= \left( \begin{matrix} 0 & \tuB(\bfY)\\ -\tuB(\bfY) & 0 \end{matrix}\right)\quad \textup{ and } \quad A(\bfX) = \left(\begin{matrix}\phantom{-}\tuBd(\bfX_1) \\ -\tuBd(\bfX_2) \end{matrix}\right).$$
This is consistent with the descriptions \cite[Rem.~7.10]{Rossmann/19} of
$B(\bfY)$ and \cite[Rem.~7.6~(ii)]{Rossmann/19} of~$A(\bfX)$. The fact that
the bottom part of $A(\bfX)$ is, up to a sign, equal to its top part (and not
its $\circ$-dual in the sense of \cite[\S~4.1]{Rossmann/19}) reflects the
symmetry of~$\tuB$; see \cite[Prop.~4.12]{Rossmann/19}.
\end{remark}

\begin{definition}\label{def:A}
The \emph{left-regular representation} of $U$ is the homomorphism
  $$\ad:U\rightarrow\Hom(W,T), \quad u\mapsto \ad_u:(w\mapsto \phi(u,w)).$$
\end{definition}

\begin{lemma}\label{lemma:techA-matrix}
  Let $u\in U$, $w\in W$, and let $D\in\Mat_d(\alg)$. Then the following hold:
\begin{enumerate}[label=$(\arabic*)$]
\item \label{adA}$\ad_u(w)=w\tuBd(u)$.
\item If $D^{\tp}\tuB=\tuB D$ holds, then $D$ stabilizes $\ker\tuB(u)$.
\end{enumerate}
\end{lemma}

\begin{proof}
  (1): Straightforward.  (2): Let $v\in\ker\tuB(u)$. Then the assumption
  implies that $\tuB(u)Dv^{\tp}=D^{\tp}\tuB(u)v^{\tp}=0$.
\end{proof}

\noindent
We next observe that matrices of linear forms that are, as the ones in
\eqref{def:BiS}, defined as Hessian matrices, have a remarkable self-duality
property.
\begin{lemma}\label{lemma:HessianAB}
  Assume that $d=3$ and let $g(y_1,y_2,y_3)\in \roc[y_1,y_2,y_3]$ be a
  homogeneous cubic polynomial.  The \emph{Hessian matrix} of $g$, viz.\
  $$\tuB_g(\bfy) = \HM(g(\bfy)) = \left( \frac{\partial g(\bfy)}{\partial y_i
      \partial y_j}\right)_{ij} \in\Mat_3(\roc[\bfy]),$$ satisfies
  $$\tuBd_g = \tuB_g = \tuB_g^{\tp}.$$
\end{lemma}
\begin{proof}
  We write
  $g(\bfy) = \sum_{\bfe\in\N_0^3, \, \sum e_i = 3} \lambda_{\bfe} \bfy^\bfe$
  and use the shorthand notation $(\alpha,\beta,\gamma)$ for the linear form
  $\alpha y_1 + \beta y_2 + \gamma y_3$. One easily computes
\begin{equation*}
  \tuB_g(\bfy) =\left( \begin{matrix} (6 \lambda_{300}, 2
      \lambda_{210}, 2 \lambda_{201})& (2 \lambda_{210},
      2\lambda_{120}, \lambda_{111})& (2 \lambda_{201}, \lambda_{111},
      2\lambda_{102})\\ (2 \lambda_{210}, 2\lambda_{120},
      \lambda_{111})&(2 \lambda_{120}, 6 \lambda_{030},
      2\lambda_{021})&(\lambda_{111}, 2 \lambda_{021}, 2
      \lambda_{012})\\(2 \lambda_{201}, \lambda_{111},
      2\lambda_{102})&(\lambda_{111}, 2 \lambda_{021}, 2 \lambda_{012})&
      (2 \lambda_{102}, 2 \lambda_{012}, 6
      \lambda_{003}) \end{matrix}\right),
\end{equation*}
from which the claim follows by inspection.
\end{proof}

\begin{definition}\label{def:vanishing}\
 \begin{enumerate}[label=$(\arabic*)$]
 \item The \emph{affine degeneracy locus of $\tuB(\bfy)$} is the closed
   subscheme $\cor{V}_\tuB$ of $\mathbb{A}^{d}$ defined by the
   equation~$\det(\tuB(\bfy)) = 0$. The set
   $\cor{V}_{\tuB}(U) = \{u\in U \mid \det(\tuB(u)) = 0\}$ is the \emph{affine
     degeneracy locus of $\tuB(\bfy)$ in $U$}.
 \item The \emph{projective degeneracy locus of $\tuB(\bfy)$} is the closed
   subscheme $\mbbP\cor{V}_{\tuB}$ of $\mathbb{P}^{d-1}$ defined by the
   equation~$\det(\tuB(\bfy)) = 0$. The set
   $\mbbP\cor{V}_{\tuB}(U)=\{u\in \mbbP U \mid \det(\tuB(u)) = 0\}$ is the
   \emph{projective degeneracy locus $\mbbP\cor{V}_{\tuB}(U)$ of~$\tuB(\bfy)$
     in $U$}.
 \end{enumerate}
\end{definition}

\begin{corollary}[to Lemma~\ref{lemma:HessianAB}]\label{cor:vanishing}
  Assume that $d=3$ and let $g(\bfy)\in \roc[\bfy]$ be a homogeneous cubic
  polynomial with Hessian matrix~$\tuB_g$. Then
  $\cor{V}_{\tuB_g} = \cor{V}_{\tuBd_g}$.
\end{corollary}
\begin{lemma}[cf.\ {\cite[Lem.~1]{VollSing}}]\label{lemma:smooth-->regular}
  Assume that $\mbbP\cor{V}_{\tuB}$ is smooth. Then, for $u\in U$, the
  following holds:
\[
  \rk\tuB(u)=\begin{cases}
    0, & \textup{if } u=0,\\
    d-1, & \textup{if }  u \in \cor{V}_\tuB(U)\setminus\{0\},\\
    d, & \textup{otherwise}.
\end{cases}
\]
\end{lemma}

\begin{remark}\label{rmk:centralizersOV15}
  Write $v\in V$ as $v=u+w$ with $u\in U$ and $w\in W$.  One checks easily
  that $\Cyc_{V}(v)=\graffe{v'\in V \mid [v,v']=0}$, the intersection of $V$
  with the centralizer $\Cyc_{\LA}(v)$ of $v$ in $\LA$, has $K$-dimension
$$\dim_{K}\Cyc_{V}(v) = 2d - \rk \left(\begin{matrix}\phantom{-}\tuBd(u) \\ -\,\tuBd(w) \end{matrix}\right);$$
cf.\ also Lemma~\ref{lemma:techA-matrix}\ref{adA} and
Remark~\ref{rem:duality}.  In particular, if $v=u$ is an element of $U$ and
$\mathbb{P}\cor{V}_{\tuBd}$ is smooth, then
$\dim_{K}\Cyc_{V}(u) = 2d-\rk\tuBd(u)$ and so, by
Lemma~\ref{lemma:smooth-->regular}, the following holds:
\begin{equation}\label{eq:CentSmooth}
  \dim_K\Cyc_{V}(u)=\begin{cases}
    2d , & \text{if } u=0, \\
    d+1 , & \text{if } u \in\cor{V}_{\tuBd}(U)\setminus\graffe{0}, \\
    d, & \text{otherwise}. 
\end{cases}
\end{equation}
If, additionally, $d=3$ and $\tuB=\tuBd$, as in the situation
described in Lemma~\ref{lemma:HessianAB}, then we find that
$\dim_{K}\Cyc_{V}(u) = 6-\rk\tuB(u)$ and hence
  \begin{equation}\label{eq:CentHessian}
  \dim_K\Cyc_{V}(u)=\begin{cases}
    6, & \text{if } u=0, \\
    4, & \text{if } u \in\cor{V}_{\tuB}(U)\setminus\graffe{0}, \\
    3, & \text{otherwise.}
    \end{cases}
\end{equation}
We remark that this applies, in particular, to the matrices $\tuB_{i,\para}$
defined in \eqref{def:BiS}
\end{remark}

\subsection{Implications for automorphism groups of
  $p$-groups}\label{subsec:imp}

\noindent
For the rest of the section we assume that $d=3$. We mainly describe,
in Proposition~\ref{prop:abel3d}, the $3$-dimensional abelian
subalgebras of $\LA$ that are contained in $V$. This allows us to
apply Lemma~\ref{lemma:Prod+Intersection}, leading to a refined
formula for $|\Aut(\la{\tuB}{\ff})|$ in the case that $\ff$ is a
finite field and $\mathbb{P}\cor{V}_{\tuB}$ and
$\mathbb{P}\cor{V}_{\tuBd}$ are {elliptic} curves; see Corollary
\ref{cor:formula}.

\begin{proposition}\label{prop:abel3d}
  Assume that $\mathbb{P}\cor{V}_{\tuB}$ and $\mathbb{P}\cor{V}_{\tuBd}$ are
  {elliptic} curves.  Then the $3$-dimensional abelian subalgebras of
  $\LA$ that are contained in $V$ are exactly those of the form $\psi(M)(U)$
  for some~$M\in\GL_2(\alg)$.
\end{proposition}

\begin{proof}
  Let $X$ be a $3$-dimensional abelian subalgebra of $\LA$ that is
  contained in~$V$.  Assume first that $X\cap U\neq\graffe{0}$. We
  claim that $X=U$ and assume, for a contradiction, that $X\neq U$.
  Then there exists $u\in X \cap U$ with centralizer $\Cyc_{V}(u)$ of
  dimension at least $4$ and such that $X$ is contained in
  $\Cyc_{V}(u)$. Then \eqref{eq:CentSmooth} implies that
  $\dim_\alg\Cyc_V(u)=4$ or, equivalently, that the kernel of $\ad_u$
  has dimension~$1$.  Let $w\in W$ be such that $\ker\ad_u= \alg w$
  and define $Y=U\oplus\ker\ad_u$ so that $Y$ has dimension $4$. Since
  $X$ and $U$ do not coincide, it follows that $Y=X+U$ and $X\cap U$
  has dimension $2$.  We let $t\in X$ be such that
  $X=(X\cap U)\oplus \alg t$ and denote by $\pi_U:V\rightarrow U$
  resp.\ $\pi_W:V\rightarrow W$ the natural projections. Since $t$
  belongs to~$Y$, we have $\pi_W(t)=\lambda t$ for some
  $\lambda \in K^\times$. It follows that
\begin{equation*}
  0=\phi(X\cap U,t)  = \phi(X\cap U,\pi_U(t)+\pi_W(t)) 
  = \phi(X\cap U,\pi_W(t)) = \lambda\phi(X\cap U, w),
\end{equation*}
which implies that $Y$ is an abelian subalgebra of $\LA$ of dimension $4$. As
a consequence, $w$ is a central element and so, by Lemma
\ref{lemma:techA-matrix}(1), the rank of $\tuBd(w)$ is zero; contradiction to
Lemma~\ref{lemma:smooth-->regular}. To conclude, in this case one can take
$M=\Id_2$ to get that $X=\psi(M)(U)=U$.

Assume now that $X\cap U=\graffe{0}$.  Then $X$ is a complement of $U$
in $V$, as $W$ is, and thus there exists $D\in\Mat_3(\alg)$ such that
$X=\graffe{w+\overline{w}D^{\tp} \mid w\in W}$. Fix such a~$D$. By
Lemma~\ref{lemma:xreduction}, it satisfies $D^{\tp}\tuB=\tuB D$. We
claim that $D$ is a scalar matrix, i.e.\ $D=\lambda\Id_3$ for some
$\lambda\in K$. For this we may, without loss of generality, assume
that $\alg$ is algebraically closed. Indeed, solving
$D^{\tp}\tuB=\tuB D$ over $\alg$ reduces to solving a system of $9$
linear equations in $9$ indeterminates and, if $D^{\tp}\tuB=\tuB D$
implies that $D$ is scalar over an algebraic closure of~$\alg$, then
the same holds over~$\alg$. Let $\cor{U}$ be a subset of
$\cor{V}_{\tuBd}(U)$ of cardinality $4$ such that any $3$ of its
elements form a basis of $U$; such $\cor{U}$ exist by
Lemma~\ref{lemma:3outof4}. By \eqref{eq:A-B} there exists, for each
$u\in\cor{U}$, an element $w_u\in W\setminus\graffe{0}$, unique up to
scalar multiplication, such that
\[
  \tuBd(u){w_u}^{\tp}=0=\tuB(w_u)u^{\tp}.
\]
In particular, combining Lemma~\ref{lemma:smooth-->regular} and \eqref{eq:A-B}
yields a well-defined bijection
\[
  \graffe{\mathbb{P}(Ku) \mid u\in\cor{U}}
  \longrightarrow\graffe{\mathbb{P}(Kw_u) \mid u\in\cor{U}}, \quad
  \mathbb{P}(Ku)\mapsto \mathbb{P}(\ker\tuBd(u))
\]
with inverse $\mathbb{P}(Kw)\mapsto \mathbb{P}(\ker\tuB(w))$. It
follows from Lemma~\ref{lemma:techA-matrix}(2) and Lemma
\ref{lemma:smooth-->regular} that each $u\in\cor{U}$ is an eigenvector
with respect to $D$ with eigenvalue $\lambda_u\in K$, say. The fact
that any three elements of $\cor{U}$ generate $U$ yields the existence
of a
$\lambda\in K$ such that for each $u\in U$ one has
$\lambda=\lambda_u$; in particular, the matrix $D$ is equal to
$\lambda\Id_3$. For such $\lambda$, defining
$M = \left( \begin{matrix} 1 & \lambda \\ 0 & 1 \end{matrix} \right)$
yields that $X=\psi(M)(U)$.
\end{proof}

\begin{corollary}\label{lemma:galois}
   Let $\ff$ be a finite field of odd characteristic and assume that
  $\mathbb{P}\cor{V}_{\tuB}$ and $\mathbb{P}\cor{V}_{\tuBd}$ are
  {elliptic} curves over~$\ff$.  Then the following holds:
  $$\Aut_{\Fp}(\la{\tuB}{\ff}) \cong \Aut(\la{\tuB}{\ff})\rtimes
  \Gal(\ff/\Fp).$$
\end{corollary}
 
\begin{proof}
  Write $\LA$ for $\la{\tuB}{\ff}$ and let $\Cyc_{\LA}$ denote the
  centroid of $\LA$, as defined in~\cite[\S~1.1]{Wilson/17}.  Using
  Proposition \ref{prop:abel3d} and extending the arguments of Lemmas
  \ref{lemma:xreduction} and Lemma~\ref{lemma:Prod+Intersection}, it
  is not difficult to see that $\Cyc_{\LA}$ is the collection of all
  the scalar multiplications on $\LA$ by elements of $\ff$ and so
  $\Cyc_{\LA}\cong \ff$. As a consequence of
  \cite[Thm.~1.2(D)]{Wilson/17} (cf.\ also \cite[\S 5]{Wilson/17}) we
  obtain an exact sequence
\[
1\rightarrow \Aut(\LA)\rightarrow\Aut_{\Fp}(\LA)\rightarrow\Gal(\ff/\Fp).
\]
The map $ \sigma \ \mapsto \left(\,
\sum_{i=1}^3(\lambda_ie_i+\mu_if_i+\nu_ig_i) \mapsto
\sum_{i=1}^3\left(\sigma(\lambda_i)e_i+\sigma(\mu_i)f_i+\sigma(\nu_i)g_i\right)\,\right)
$ being a section $\Gal(\ff/\Fp)\rightarrow\Aut_{\Fp}(\LA)$ it is, in
fact, a split short exact sequence. The result follows.
%We have proven that our original sequence extends on the right to a short
%exact semi-split sequence and thus that
%$\Aut_{\F_p}(\LA)\cong \Aut(\LA)\rtimes \Gal(\ff/\Fp)$.
\end{proof} 

\begin{remark}
  Proposition~\ref{prop:abel3d} generalizes the computations
  in~\cite[\S 7]{dS+VL} for $\tuB=\tuB_{1,1}$ defined
  in~\eqref{def:BiS}.
\end{remark}

\noindent
As mentioned in Section~\ref{sec:groups.from.forms},
Proposition~\ref{prop:abel3d} allows us to apply
Lemma~\ref{lemma:Prod+Intersection} in the proof of
Theorem~\ref{th:main.para}. This lemma reduces the determination of the order
of the automorphism group of $\LA_\tuB$ to the analysis of the structure of
the subgroup $\Aut_V^=(\LA_\tuB)$ of~$\Aut(\LA_\tuB)$.  In the next
proposition, let $\pi:\Aut_V^=(\la{\tuB}{\ff})\rightarrow \PGL_3(\ff)$ be the
map defined by $\diag(A,A,A_T)\mapsto \overline{A}$, where $\overline{A}$ is
the image in $\PGL_3(\alg)$ of $A$.

\begin{proposition}\label{rmk:existence of c}
  Assume that $\mathbb{P}\cor{V}_{\tuBd}$ is an {elliptic}
  curve. Then the map
\begin{equation*}
  \overline{c}_{\tuB}:  \Aut_V^=(\LA_{\tuB})\longrightarrow \Aut(\mathbb{P}\cor{V}_{\tuBd}),\quad
  \diag(A,A,A_T) \longmapsto \overline{A}_{|\cor{V}_{\tuBd}},
\end{equation*}
is a well-defined
homomorphism of groups and satisfies
$\overline{c}_{\tuB}=\overline{c}_{\mathbb{P}\cor{V}_\tuBd}\circ \pi$.
\end{proposition}

\begin{proof}
  The elements of $\Aut_V^=(\LA_\tuB)$ stabilize the abelian subalgebra $U$
  and, being Lie algebra homomorphisms, respect the centralizer dimensions
  $\dim \Cyc_V(u)$ of the elements $u\in U$.  More concretely, if $\varphi$
  belongs to $\Aut_V^=(\LA_\tuB)$, then, thanks to
  Remark~\ref{rmk:centralizersOV15}, the following hold:
\[
  \varphi(0)=0,\quad
  \varphi(\cor{V}_{\tuBd}(U)\setminus\graffe{0})=\cor{V}_{\tuBd}(U)\setminus\graffe{0},
  \quad \varphi(U\setminus\cor{V}_{\tuBd}(U))=U\setminus \cor{V}_{\tuBd}(U).
\]
In particular, any element of $\Aut_V^=(\LA_\tuB)$ induces a bijection
$\cor{V}_{\tuBd}(U)\setminus\graffe{0}\longrightarrow\cor{V}_{\tuBd}(U)\setminus\graffe{0}$
and so projectivisation yields the homomorphism $\overline{c}_{\tuB}$.
\end{proof}

\noindent
The idea of appealing to the subgroup $\Aut_V^=(\LA)$ in order to
determine the structure of $\Aut(\LA)$ was already pursued in
\cite{dS+VL}.  This notwithstanding, the map~$\overline{c}_{\tuB}$
holds the key to the realization of the added value brought about by
our geometric point of view. Indeed, the determination of the image of
$\overline{c}_{\tuB}$ plays a decisive role in the proof of
Theorem~\ref{th:main.para} by means of the following corollary.

\begin{corollary}\label{cor:formula}
  Let $\ff$ be a finite field of odd characteristic and assume that
  $\mathbb{P}\cor{V}_{\tuB}$ and $\mathbb{P}\cor{V}_{\tuBd}$ are
  {elliptic} curves over~$\ff$.  Then the following hold:
\begin{equation*}\label{eq:AutLcard}
  |\Aut(\la{\tuB}{\ff})| =\frac{|\Aut_V^=(\la{\tuB}{\ff})|\cdot|\GL_2(\ff)|\cdot |\ff|^{18}}{|\ff|-1} =|\im\overline{c}_{\tuB}|\cdot |\GL_2(\ff)|\cdot |\ff|^{18}.
\end{equation*}
\end{corollary}

\begin{proof}
  The combination of \eqref{eq:AutL}, Proposition~\ref{prop:abel3d},
  and Lemma~\ref{lemma:Prod+Intersection}\ref{fourth.new} yields that
\begin{equation*}
|\Aut(\la{\tuB}{\ff})| =\frac{|\Aut_V^=(\la{\tuB}{\ff})|\cdot|\GL_2(\ff)|\cdot |\ff|^{18}}{|\ff|-1}%\\& 
=\frac{|\ker\overline{c}_{\tuB}|\cdot|\im\overline{c}_{\tuB}|\cdot|\GL_2(\ff)|\cdot |\ff|^{18}}{|\ff|-1}.
\end{equation*}
The claim follows as Lemmas~\ref{lemma:barcgeneral}
and~\ref{lemma:Prod+Intersection}\ref{third.new} imply that
$\ker\overline{c}_{\tuB}=\graffe{\lambda\Id_9 \mid
  \lambda\in\ff^\times} \cong \ff^\times$.
\end{proof}

\section{Proofs of the main results and their corollaries}\label{sec:pfs}

We prove Theorem~\ref{th:main.para} and its
Corollaries~\ref{cor:non-PORC}-\ref{cor:descendants} in
Section~\ref{subsec:auto} and Theorem~\ref{th:iso} in
Section~\ref{subsec:iso}, where we also prove Theorem~\ref{th:intro}.

\subsection{Automorphisms}\label{subsec:auto}
We continue to use the setup from Section~\ref{sec:lifting} and
combine it with that from Section~\ref{subsec:setup}. Recall
that $\para$ is a nonzero integer and that $\alg$ is a field of
characteristic not dividing $2\para$. Assume further that $\alg$
contains a fixed square root $\sqrtpara$ of $\para$.  For
$i\in\{1,2,3\}$, let $\tuB_{i,\para}$ be as in~\eqref{def:BiS}, and
denote by $\phi_{i,\para}$ and~$\widetilde{\phi}_{i,\para}$,
respectively, the associated bilinear and linear maps defined in
\eqref{def:phi} resp.~\eqref{def:phi.tilde}. The image of
$\phi_{i,\para}$ spans $T$ so $\widetilde{\phi}_{i,\para}$ is
surjective.

Write ${\bf G}_{i,\para}={\bf G}_{\tuB_{i,\para}}$ and
$\LA_{i,\para} = \LA_{\tuB_{i,\para}}$, respectively, for the group
and the Lie algebra (schemes) associated with the data
$(K,B_{i,\para},\phi_{i,\para})$ in Section~\ref{subsec:groups}.  We
recall that, if $\alg = \ff$ is a finite field of order $q$ and odd
characteristic~$p$, then the finite $p$-group $\gp{{i,\para}}{\ff}$
has exponent~$p$, nilpotency class~$2$, and order~$q^9$, while the
$\ff$-Lie algebra $\LA_{i,\para}$ has $\ff$-dimension~$9$ and
nilpotency class~$2$.  We are looking to compute the order of
$\Aut(\la{i,\para}{\ff})$. For this, we will use
Corollary~\ref{cor:formula}.  Indeed, the matrix $\tuB_{i,\para}$ is
Hessian and therefore satisfies $\tuB_{i,\para}=\tuBd_{i,\para}$; see
Lemma~\ref{lemma:HessianAB}.  Observe that
$\mathbb{P}\cor{V}_{\tuB_{i,\para}}$ is identified with the {elliptic}
curve $E_{\para}$ via the projectivisation~\eqref{eq:projcurvepara}.
Let
\[
  \overline{c}_{\tuB_{i,\para}}:\Aut^=_V(\LA_{i,\para})\rightarrow\Aut(E_{\para})
\] 
be the homomorphism from Proposition~\ref{rmk:existence of c}. By
Proposition~\ref{prop:lifting IsomE.para}, its image is isomorphic to
a subgroup of $E_{\para}[3]\rtimes\Aut_{\neec}(E_{\para})$,
which leads us to consider the homomorphism
\begin{align}\label{def:c.bar}
  \overline{c}_{i,\para}:\ \Aut_V^=(\LA_{i,\para})&\longrightarrow E_{\para}[3]\rtimes\Aut_{\neec}(E_{\para}),\nonumber\\
  \quad \diag(A,A,A_T) &\longmapsto \overline{c}_{\tuB_{i,\para}}(\diag(A,A,A_T))=(Q_A,\alpha_A).
\end{align}
Corollary~\ref{cor:formula} reduces the proof of Theorem~\ref{th:main.para} to
the explicit determination of the image size
of~$\overline{c}_{i,\para}$. To this end, we will make use of the
following specific version of Lemma~\ref{lemma:kernel&tensors}.

\begin{lemma}\label{obs:kernel}
  Let $A\in\GL_3(\alg)$ and define the
  following subsets of $U\otimes W$:
\begin{align*}
\cor{K}_{1,\para}^*= &\graffe{e_2\otimes f_3,
\para e_1\otimes f_1-e_3\otimes f_3
,e_1\otimes f_3+e_2\otimes f_2}, \\
\cor{K}_{2,\para}^*=&\graffe{(\sqrtpara e_1+e_3)\otimes f_3,e_2\otimes f_3-\sqrtpara e_1\otimes f_2, \para e_1\otimes f_1+2\sqrtpara e_2\otimes f_2+ e_3\otimes f_3}, \\
\cor{K}_{3,\para}^*=&\graffe{(\sqrtpara e_1-e_3)\otimes f_3,e_2\otimes f_3+\sqrtpara e_1\otimes f_2, \para e_1\otimes f_1-2\sqrtpara e_2\otimes f_2+ e_3\otimes f_3}.
\end{align*}
Then the following are equivalent:
\begin{enumerate}[label=$(\arabic*)$]
\item
  $(A\otimes A)(\cor{K}_{i,\para}^*)\subseteq\ker\widetilde{\phi}_{i,\para}$;
\item there exists $A_T\in\GL_3(\alg)$ such that
  $\diag(A,A,A_T)\in\Aut(\LA_{i,\para})$.
\end{enumerate}
\end{lemma}

\begin{proof}
  Consider the following supersets of $\cor{K}^*_{i,\para}$ of
  $\cor{K}_{i,\para}$ in $U\otimes W$:
\begin{alignat*}{1}
% \cor{K}_{1,\para}  =\graffe{& e_2\otimes f_3, e_3\otimes f_2,  
%                \para e_1\otimes f_1-e_3\otimes f_3, \\
%               & e_1\otimes f_3-e_3\otimes f_1, e_1\otimes f_2-e_2\otimes f_1,  e_1\otimes f_3+e_2\otimes f_2}, \\
 \cor{K}_{1,\para}  &=\cor{K}^*_{1,\para}\cup \graffe{e_3\otimes f_2, e_1\otimes f_3-e_3\otimes f_1, e_1\otimes f_2-e_2\otimes f_1}, \\
% \cor{K}_{2,\para} =\graffe{& (\sqrtpara e_1+e_3)\otimes f_3, e_3\otimes(\sqrtpara f_1+f_3), e_2\otimes f_3-\sqrtpara e_1\otimes f_2, \sqrtpara e_2\otimes f_1-e_3\otimes f_2, \\
% 			  & e_2\otimes f_3-e_3\otimes f_2, \para e_1\otimes f_1+2\sqrtpara e_2\otimes f_2+ e_3\otimes f_3}, \\
 \cor{K}_{2,\para} &=\cor{K}^*_{2,\para} \cup \graffe{e_3\otimes(\sqrtpara f_1+f_3), \sqrtpara e_2\otimes f_1-e_3\otimes f_2, e_2\otimes f_3-e_3\otimes f_2}, \\
% \cor{K}_{3,\para} =\graffe{& (\sqrtpara e_1-e_3)\otimes f_3, e_3\otimes(\sqrtpara f_1-f_3), e_2\otimes f_3+ \sqrtpara e_1\otimes f_2, \sqrtpara e_2\otimes f_1+e_3\otimes f_2, \\
% 			  & e_2\otimes f_3-e_3\otimes f_2, \para e_1\otimes f_1-2\sqrtpara e_2\otimes f_2+ e_3\otimes f_3}\\
 \cor{K}_{3,\para} &=\cor{K}^*_{3,\para} \cup \graffe{e_3\otimes(\sqrtpara f_1-f_3), \sqrtpara e_2\otimes f_1+e_3\otimes f_2, e_2\otimes f_3-e_3\otimes f_2}.
\end{alignat*}
The kernel of $\widetilde{\phi}_{i,\para}$ is spanned by
$\cor{K}_{i,\para}$ over $\alg$.  It follows, however, from
Remark~\ref{rk:dual} that to check
whether $(A\otimes A)(\ker\widetilde{\phi}_{i,\para})$ is contained in
$\ker\widetilde{\phi}_{i,\para}$ it suffices to check if~$(A\otimes A)(\cor{K}_{i,\para}^*)$ is annihilated by
$\widetilde{\phi}_{i,\para}$.
%, where
%\begin{align*}
%\cor{K}_{1,\para}^*= &\graffe{e_2\otimes f_3,
%\para e_1\otimes f_1-e_3\otimes f_3
%,e_1\otimes f_3+e_2\otimes f_2}, \\
%\cor{K}_{2,\para}^*=&\graffe{(\sqrtpara e_1+e_3)\otimes f_3,e_2\otimes f_3-\sqrtpara e_1\otimes f_2, \para e_1\otimes f_1+2\sqrtpara e_2\otimes f_2+ e_3\otimes f_3}, \\
%\cor{K}_{3,\para}^*=&\graffe{(\sqrtpara e_1-e_3)\otimes f_3,e_2\otimes f_3+\sqrtpara e_1\otimes f_2, \para e_1\otimes f_1-2\sqrtpara e_2\otimes f_2+ e_3\otimes f_3}.
%\end{align*}
Indeed, the elements of
$\cor{K}_{i,\para}\setminus\cor{K}_{i,\para}^*$ are equal to the
negatives of their respective duals or duals to members of
$\cor{K}_{i,\para}^*$. To conclude we apply
Lemma~\ref{lemma:kernel&tensors}.
\end{proof}

\begin{lemma}\label{lemma:IsogenyRespectGP.new}
  Let $\alpha\in\Aut_{\neec}(E_\para)$. The following are equivalent:
  \begin{enumerate}[label=$(\arabic*)$]
  \item There exist $A,A_T\in\GL_3(\alg)$ such that
    $\overline{c}_{i,\para}(\diag(A,A,A_T))=(\neec,\alpha)$;
  \item $\alpha^{\lceil 4/i \rceil} = \id_{E_{\para}}$.
  \end{enumerate}
\end{lemma}

\begin{proof}
  Let $(\omega,\rho)$ be such that $\omega^4=1$ and $\para\rho^3=\rho$
  with $\rho\neq 0$ only if $\char(\alg)=3$. Let
  $\alpha\in\Aut_{\neec}(E_{\para})$ be the invertible isogeny defined by
  $(x,y)\mapsto(\omega^2x+\rho,\omega^3y)$ and observe that, if
  $i\in\graffe{2,3}$, then $\lceil 4/i \rceil=2$.
  By~Lemma~\ref{lemma:isogenieslift.para}, the matrix
$$A=\begin{pmatrix}
\omega^2 & 0 & \rho \\
0 & \omega^3 & 0 \\
0 & 0 & 1
 \end{pmatrix}$$
is the unique element of $\GL_3(\alg)$, up to scalar multiplication, inducing $\alpha$. 

We start by considering $i=1$. A necessary condition for
$(A\otimes A)(\cor{K}_{1,\para}^*)$ to be contained in
$\ker\widetilde{\phi}_{1,\para}$ is that
$0=\widetilde{\phi}_{1,\para}((A\otimes A)(e_1\otimes f_3+e_2\otimes
f_2))=%\omega^2\widetilde{\phi}_{1,\para}(\rho e_1\otimes f_1+e_1\otimes f_3+e_2\otimes f_2 )=
\omega^2 \rho g_3$,
equivalently $\rho=0$. It is not difficult to show, for $\rho=0$, that
$(A\otimes A)(\cor{K}_{1,\para}^*)$ is contained in
$\ker\widetilde{\phi}_{1,\para}$ and so, by Lemma~\ref{obs:kernel},
there exists $A_T\in\GL_3(K)$ such that $\diag(A,A,A_T)$ is an
automorphism of $\la{1,\para}{K}$. By the definition of
$\overline{c}_{1,\para}$, the pair $(\neec,\alpha)$ is contained in
its image.

Let now $i\in\{2,3\}$. Computing
$(A\otimes A)(e_2\otimes f_3\mp \sqrtpara e_1\otimes f_2)$ for $i=2$ resp.\
$i=3$ we get that
\[
  (A\otimes A)(e_2\otimes f_3\mp \sqrtpara e_1\otimes f_2)\in\ker\widetilde{\phi}_{i,\para} \ \ \Longleftrightarrow
  \ \ \omega^2\rho\pm \omega^2\sqrtpara \mp \sqrtpara=0.
\]
If $\rho=0$, one sees that $\omega^2=1$ is a necessary and sufficient
condition for $(A\otimes A)(\cor{K}_{i,\para}^*)$ to be contained in
$\ker\widetilde{\phi}_{i,\para}$. Assume, in conclusion, that $\rho\neq 0$ and
thus that $\char(\alg)=3$: we show that
$\alpha\notin \im\overline{c}_{i,\para}$. We work by contradiction, assuming
that
$\widetilde{\phi}_{i,\para}((A\otimes A)(\cor{K}_{i,\para}^*)) =
\graffe{0}$. It follows then that $\para\rho^2=1$ and so we find that
\[
1=\para\rho^2=\para^2(\pm 1\mp\omega^2)^2\omega^{-4}=2\para^2(1-\omega^2).
\]
In particular, $\omega^2=-1$ and so that $\rho=\mp 2 \sqrtpara$.
From
$$\widetilde{\phi}_{i,\para}(A\otimes A)(\para e_1\otimes f_1\pm
2\sqrtpara e_2\otimes f_2+ e_3\otimes f_3)=0$$
we then get that
$0=(\pm \sqrtpara+2\rho)g_1+(\pm 2\rho\sqrtpara +\para)g_3$ and so,
from the coefficient of $g_3$, we derive that
$\para=\pm2\rho\sqrtpara=-4\para=-\para$. Contradiction. We conclude
with Lemma~\ref{obs:kernel}.
\end{proof}

\begin{lemma}\label{lemma:TranslationRespectGP.new}
  Let $Q=(a,b)$ be an element of $E_{\para}[3]$. Then there exist
  $A,A_T\in\GL_3(\alg)$ such that
  $$\overline{c}_{i,\para}(\diag(A,A,A_T))=(Q,\id_E).$$
  %if and only if $S \equiv 1 \bmod p$.
\end{lemma}

\begin{proof}
  Up to scalar multiplication, Lemma~\ref{lemma:oneil} yields that a
  necessary condition for the existence of such a pair $(A,A_T)$ is
  the existence of $\nu\in K^{\times}$ such that
\[
A=A(\nu)=\begin{pmatrix}
\para ab+2ab\nu & \para a^2 & -2\para a^2b\nu \\
(-3\para-2\nu)b^2 & \para ab & 2\para ab^2\nu  \\
(1-2\nu a^2)b & a & 2\para a^3b\nu 
\end{pmatrix}.
\]
We set $\nu= -\para/(\para a^2+1)$ and claim that, for this matrix
$A=A(\nu)$, there exists $A_{T}=A_T(\nu)$ such that
$\overline{c}_{i,\para}(\diag(A,A,A_T))$ is equal to~$(Q,\id_E).$
Indeed, checking (e.g.\ with SageMath \cite{SageMath}) identities
involving the matrix $C = A^{\tp} \tuB_{i,\para}A$ defined in
Remark~\ref{rk:dual}, one shows
that~$(A\otimes
A)(\cor{K}_{i,\para}^*)\subseteq\ker\widetilde{\phi}_{i,\para}$.
We conclude with Lemma~\ref{obs:kernel}.
\end{proof}

\begin{corollary}\label{cor:main.para}
  Let $\ff$ be a finite field of characteristic not dividing~$2\para$ in which
  $\para$ has a fixed square root.  Then the following holds:
  $$|\Aut(\la{i,\para}{\ff})| = \GCD{|\ff|-1, \ceil*{4/i}}
  |\GL_2(\ff)|\cdot |\ff|^{18} \cdot |E_{\para}[3](\ff)|.$$
\end{corollary}

\begin{proof}
  Combining Lemmas~\ref{lemma:IsogenyRespectGP.new}
  and~\ref{lemma:TranslationRespectGP.new} allows us to describe the
  image of the homomorphism $\overline{c}_{i,\para}$
  (cf.~\eqref{def:c.bar}) and hence, by Corollary~\ref{cor:formula},
  the order of~$\Aut(\la{i,\para}{\ff})$.
\end{proof}

\noindent
To prove Theorem~\ref{th:main.para} it now suffices to note that
$|\Aut(\gp{i,\para}{\ff})| = |\Aut_{\F_p}(\la{i,\para}{\ff})|$ by the
Bear correspondence. We may thus conclude by combining
Corllary~\ref{lemma:galois} and and Lemma~\ref{cor:main.para}.
\vspace{5pt}\\
\noindent
We conclude by proving
Corollaries~\ref{cor:non-PORC}-\ref{cor:descendants}. To this end,
 let $p$ be a prime not dividing $6\para$ and let
$\mathrm{n}_{1,1}(p)$ denote the number of immediate descendants of
$\gp{1,1}{\Fp}$ of order $p^{10}$ and exponent $p$. In
\cite[Sec.~4]{dS+VL} it is shown that neither of the two functions
\[p\longmapsto |\Aut(\gp{1,1}{\F_p})|, \quad p\longmapsto
\mathrm{n}_{1,1}(p)\]
are PORC, the second being so as a consequence of the first. Combining
Theorem~\ref{th:main} with Lemma~\ref{lemma:3torEQ.para}, this amounts
to saying that the function
$\Pi\rightarrow\Z,\, p\mapsto |E_1[3](\Fp)|$ is not  constant on residue classes modulo a fixed
integer. Corollary~\ref{cor:non-PORC} now follows from
Theorem~\ref{th:main.para} and Lemma~\ref{lem:dSVL.torsion.points}. To
prove Corollary~\ref{cor:POFS}, observe that the function
$p\mapsto |E[3](\Fp)|$ is constant on Frobenius sets for any elliptic curve $E$ defined over $\Q$. To prove
Corollary~\ref{cor:descendants} one may proceed as in
\cite[Sec.~11]{dS+VL}, i.e.\ by counting orbits of the induced action
of $\im\overline{c}_{1,1}$ on $\F_p^3$ by means of Burnside's
lemma. One checks easily that our formula for the descendants matches
the values listed in \cite[Sec.~5]{Vaughan-Lee/12}.

\subsection{Isomorphisms}\label{subsec:iso}

\noindent
While we focussed so far on automorphisms of groups and Lie algebras
of the form described in Section~\ref{subsec:groups}, we now turn to
isomorphisms between such objects. The section's main aim is to prove
Theorem~\ref{th:iso}. To this end, we state and work with a number of
results that have close counterparts in
Section~\ref{sec:groups.from.forms}. Their proofs being entirely
analogous, we omit most of them. We also restrict to the case that
$\char(\alg)\neq 3$, to lighten notation.

Let $i,j\in\graffe{1,2,3}$ and
$\para,\para'\in\Z\setminus\graffe{0}$. Assume that the characteristic
of $\alg$ does not divide $2\para\para'$ and that $\alg$ contains
square roots $\sqrtpara$ and $\para'^{\frac{1}{2}}$ of $\para$ resp.\
$\para'$, which we fix.  We write $U_{i,\para}$, $W_{i,\para}$,
$T_{i,\para}$, $V_{i,\para}$ for the modules $U,W,T,V$ from
Section~\ref{subsec:setup}. In particular,
$U_{i,\para}\oplus W_{i,\para}\oplus T_{i,\para}$ is the underlying
vector space of $\la{i,\para}{K}$. In a similar fashion, we will write
$$(e_1^{(i,\para)},e_2^{(i,\para)}, e_3^{(i,\para)}), \quad (f_1^{(i,\para)},f_2^{(i,\para)}, f_3^{(i,\para)}), \quad (g_1^{(i,\para)},g_2^{(i,\para)}, g_3^{(i,\para)})$$
for the bases $\cor{E}=\cor{E}^{(i)}$, $\cor{F}=\cor{F}^{(i)}$,
$\cor{T}=\cor{T}^{(i)}$ from Section~\ref{subsec:setup}.  Analogous
indexing pertains to the pair~$(j,\para')$. With respect to these
bases, we may view Lie algebra isomorphisms
$\la{i,\para}{\alg}\rightarrow\la{j,\para'}{\alg}$ as matrices in
$\GL_9(\alg)$. With these identifications, an isomorphism
$\alpha:\la{i,\para}{\alg}\rightarrow\la{j,\para'}{\alg}$ satisfying
$\alpha(U_{i,\para})=U_{j,\para'}$,
$\alpha(W_{i,\para})=W_{j,\para'}$, and
$\alpha(T_{i,\para})=T_{j,\para'}$, is of the form
$\alpha=\diag(A_U,A_W,A_T)$ for some $A_U,A_W,A_T\in\GL_3(\alg)$.

\begin{lemma}\label{lemma:neccondIsoGPs}
  If $\la{i,\para}{\alg}\cong \la{j,\para'}{\alg}$, then there exists
  an isomorphism
  $\alpha:\la{i,\para}{\alg}\rightarrow\la{j,\para'}{\alg}$ such that
  the following hold:
  \begin{enumerate}[label=$(\arabic*)$]
  \item $\alpha(U_{i,\para})=U_{j,\para'}$,
    $\alpha(W_{i,\para})=W_{j,\para'}$, and
    $\alpha(T_{i,\para})=T_{j,\para'}$.
  \item The map $\alpha$ induces isomorphisms of projective curves
    $\alpha_U,\alpha_W:E_{\para}\rightarrow E_{\para'}$.
\end{enumerate}
\end{lemma}

\begin{proof}
  Let $\alpha':\la{i,\para}{\alg}\rightarrow\la{j,\para'}{\alg}$ be an
  isomorphism and assume, without loss of generality, that
  $\alpha'(V_{i,\para})=V_{j,\para'}$. Then $\alpha'(U_{i,\para})$ and
  $\alpha'(W_{i,\para})$ are $3$-dimensional abelian subalgebras of
  $\la{j,\para'}{\alg}$ that are contained in $V_{j,\para'}$. Let
  $M\in\GL_2(\alg)$ be such that
  $\psi(M)(\alpha'(U_{i,\para}))=U_{j,\para'}$ and
  $\psi(M)(\alpha'(W_{i,\para}))=W_{j,\para'}$; cf.~Proposition~\ref{prop:abel3d}. Define
  $\alpha=\psi(M)\circ\alpha'$. By \eqref{eq:CentHessian}, the Lie
  algebra isomorphism $\alpha$ induces isomorphisms of projective
  curves $\alpha_U,\alpha_W: E_{\para}\rightarrow E_{\para'}$,
  corresponding to the restrictions of $\alpha$ to $U_{i,\para}$
  resp.~$W_{i,\para}$.
\end{proof}

\noindent
The following is a variation of Lemma~\ref{lemma:xreduction}. We omit
the analogous proof.

\begin{lemma}\label{lemma:xredISO}
  Let $A_U$ and $A_W$ be elements of $\GL_3(\alg)$ and set
  $D=A_UA_W^{-1}$. Then the conditions
\begin{enumerate}[label=$(\arabic*)$]
\item there exists $A_T\in\GL_3(\alg)$ such that
  $\diag(A_U,A_W,A_T):\la{i,\para}{\alg}\rightarrow\la{j,\para'}{\alg}$ is an isomorphism;
 \item the equality $D^{\tp}\tuB_{j,\para'} = \tuB_{j,\para'} D$ holds;
 \item the subspace
   $X=\graffe{w+\overline{w}D^{\tp} \mid w\in W_{j,\para'}}$ is a
   complement of $U_{j,\para'}$ in $V_{j,\para'}$ satisfying
   $[X,X]=0$;
\end{enumerate}
are related in the following way:
\[
(1) \ \ \Longrightarrow \ \ (2) \ \ \Longleftrightarrow \ \ (3).
\]
\end{lemma}

\begin{lemma}\label{lemma:tentative}
  If $\la{i,\para}{\alg}\cong \la{j,\para'}{\alg}$, then there exists
  an isomorphism
  $\beta:\la{i,\para}{\alg}\rightarrow\la{j,\para'}{\alg}$ and
  $A,A_T\in\GL_3(\alg)$ such that the following hold:
\begin{enumerate}[label=$(\arabic*)$] 
 \item $\beta=\diag(A,A,A_T)$.
 \item There exists $\varepsilon\in\alg$ such that
   $\para=\varepsilon^4\para'$ and $A$ induces the invertible isogeny
%\footnote{MS: I fear that here we are forgetting about $p=3$. Indeed we are, checked in Silverman. Suggestion: since this section is anyway omitting things, let us say in intro that in this section we prove the results with the additional assumption that $p\neq 3$ just to have lighter arguments -- computations are analogue to those in proof of Lemma~\ref{lemma:IsogenyRespectGP.new}}
 $$E_{\para}\longrightarrow E_{\para'}, \quad (x,y)\longmapsto (\varepsilon^2x,\varepsilon^3y).$$
\end{enumerate}
\end{lemma}

\begin{proof}
  Let $\alpha$ be as in Lemma~\ref{lemma:neccondIsoGPs} and write
  $\alpha=\diag(A_U,A_W,A_T)$ for matrices
  $A_U,A_W,A_T$ in $\GL_3(\alg)$. Let, moreover, $\alpha_U,\alpha_W$
  denote the projective curve isomorphisms
  $E_{\para}\rightarrow E_{\para'}$ induced by $A_U$ resp.~$A_W$.

  (1): By Proposition~\ref{prop:abel3d}, every $3$-dimensional abelian
  subalgebra of $\la{j,\para'}{\alg}$ that is contained in $V_{j,\para'}$ is
  of the form $\psi(M)(U_{j,\para'})$ for some
  $M\in\GL_2(\alg)$. Equivalently, if $X$ is such a subalgebra, then either
  $X=U_{j,\para'}$ or
  $X=\graffe{w+\lambda\overline{w} \mid w\in W_{j,\para'}}$ for some
  $\lambda\in\alg$. The matrix $D=A_UA_W^{-1}$ being invertible,
  Lemma~\ref{lemma:xredISO} yields the existence of a unique
  $\lambda_0\in\alg^\times$ such that~$A_U=\lambda_0 A_W$. We conclude by
  defining $\beta=\psi(\diag(\lambda_0^{-1}, 1))\circ\alpha$.

  (2): Let $\beta$ be as in (1) and set $P=\alpha_U(\neec)$. Then the
  map $\alpha=\tau_{-P}\circ\alpha_U$ is an invertible
  isogeny
  $E_{\para}\rightarrow E_{\para'}$ which satisfies
  $\tau_{P}=\alpha_U\circ\alpha^{-1}$. The assumptions and
  Lemma~\ref{lemma:isogenieslift.para} imply that $\tau_P$ is induced
  by an element of $\PGL_3(\alg)$ and so, by the combination of
  Lemma~\ref{lemma:necesLIFT.para} with Lemma~\ref{lemma:3torEQ.para},
  that the point $P$ has order dividing $3$. By
  Lemma~\ref{lemma:TranslationRespectGP.new} we may choose matrices
  $A',A'_T\in\GL_3(\alg)$ with the property that
  $\diag(A',A',A'_T)\in\Aut(\la{j,\para'}{\alg})$ and
  $\overline{c}_{j,\para'}(\diag(A',A',A'_T))=(P,\id_E)$. To conclude,
  define $\beta'=\diag(A',A',A'_T)^{-1}\circ\beta$, which induces an
  invertible isogeny $\alpha':E_{\para}\rightarrow E_{\para'}$. In
  particular, there exists $\varepsilon\in\alg$ satisfying
  $\para=\varepsilon^4\para'$ and 
  $\alpha'(x,y)=(\varepsilon^2x,\varepsilon^3y)$.
\end{proof}

\noindent
The following is a variation of Lemma~\ref{lemma:kernel&tensors}. We
omit the analogous proof.

\begin{lemma}\label{lemma:kernel&tensorsII}
  Let $A$ be an element of $\GL_3(\alg)$. Then the following are
  equivalent:
  \begin{enumerate}[label=$(\arabic*)$]
  \item $(A\otimes A)(\ker\tilde{\phi}_{i,\para})\subseteq\ker\tilde{\phi}_{j,\para'}$;
  \item there exists $A_T\in\GL_3(\alg)$ such that
    $\diag(A,A,A_T):\la{i,\para}{\alg}\rightarrow\la{j,\para'}{\alg}$
    is an isomorphism.
  \end{enumerate}
\end{lemma}

\begin{lemma}\label{lemma:Aunique}
  Let $\beta=diag(A,A,A_T)$ and $\beta'=\diag(A',A',A_T')$ be Lie algebra
  isomorphisms $\la{i,\para}{\alg}\rightarrow\la{j,\para'}{\alg}$ such
  that $A$ and $A'$ induce the same invertible isogeny
  $E_{\para}\rightarrow E_{\para'}$. Then $\overline{A}$ and
  $\overline{A'}$ are equal in $\PGL_3(\alg)$.
\end{lemma}

\begin{proof}
  Since $\beta^{-1}\circ\beta'$ is an automorphism of
  $\la{i,\para}{\alg}$ inducing the identity on $E_{\para}$,
  Lemma~\ref{lemma:barcgeneral} yields the claim.
\end{proof}

\begin{proposition}\label{prop:i<j}
  Assume that $i<j$.  The following are equivalent:
\begin{enumerate}[label=$(\arabic*)$]
\item The Lie algebras $\la{i,\para}{\alg}$ and $\la{j,\para'}{\alg}$
  are isomorphic.
\item The equality $\para=\para'$ holds in $\alg$, $\graffe{i,j}=\graffe{2,3}$, and
  either \begin{enumerate}[label=$(\alph*)$]
  \item $\sqrtpara=-\para'^{\frac{1}{2}}$ or
  \item $\sqrtpara=\para'^{\frac{1}{2}}$ and $\alg$ contains a primitive $4$th root of unity.
\end{enumerate}
\end{enumerate}
\end{proposition}

\begin{proof}
  $(2)\Rightarrow(1)$: If $\para'^{\frac{1}{2}}=-\para^{\frac{1}{2}}$,
  it is obvious from inspection of~\eqref{def:BiS} that
  $\la{i,\para}{\alg}\cong\la{j,\para'}{\alg}$. Assume thus that
  $\para'^{\frac{1}{2}}=\para^{\frac{1}{2}}$ and let
  $\omega\in \mu_4(\alg)$ have order $4$. Then the map
\[
\alpha=\diag(-1,-\omega,1,-1,-\omega,1, -1, \omega, 1)
\]
is an isomorphism $\la{2,\para}{\alg}\rightarrow\la{3,\para'}{\alg}$.

$(1)\Rightarrow(2)$: Let
$\beta=\diag(A,A,A_T):\la{i,\para}{\alg}\rightarrow\la{j,\para'}{\alg}$
be an isomorphism as in Lemma~\ref{lemma:tentative}, let
$\varepsilon\in\alg$ be such that $\para=\varepsilon^4\para'$, and
assume that $A$ induces the invertible isogeny
$E_{\para}\rightarrow E_{\para'}, \, (x,y)\mapsto
(\varepsilon^2x,\varepsilon^3y)$.
By Lemma~\ref{lemma:Aunique}, the matrix $A$ equals
$\diag(\varepsilon^2,\varepsilon^3,1)$, up to a scalar.

Assume first, for a contradiction, that $i=1$.  By
Lemma~\ref{lemma:kernel&tensorsII}, the element
$e_2^{(1,\para)}\otimes f_3^{(1,\para)}$, belonging to $\ker\tilde{\phi}_{1,\para}$,
satisfies
\[
0=\tilde{\phi}_{j,\para'}\bgr{(A\otimes A)(e_2^{(1,\para)}\otimes
  f_3^{(1,\para)})}=\tilde{\phi}_{j,\para'}\left(\varepsilon^3e_2^{(j,\para')}\otimes
  f_3^{(j,\para')}\right)\in\graffe{\pm\varepsilon^3\para'^{\frac{1}{2}}g_2^{(j,\para')}};
\]
contradiction. 

Assume now that $(i,j)=(2,3)$ and let $\cor{K}_{2,\para}^*$ and
$ \cor{K}_{3,\para'}^*$ be as in Lemma~\ref{obs:kernel}. Observe that
$\cor{K}_{2,\para}^*$ is contained in $\ker\tilde{\phi}_{2,\para}$
and, analogously, $\cor{K}_{3,\para'}^*$ is contained in
$\ker\tilde{\phi}_{3,\para'}$. Lemma~\ref{lemma:kernel&tensorsII}
implies that
$\tilde{\phi}_{3,\para'}\left((A\otimes
  A)(\cor{K}_{2,\para}^*)\right)=\graffe{0}$
and so the span of
$\tilde{\phi}_{3,\para'}(\cor{K}_{2,\para}^*)\cup\cor{K}_{3,\para'}^*$
in $U_{3,\para'}\otimes W_{3,\para'}$ is contained
in~$\ker\tilde{\phi}_{3,\para'}$.  As a consequence, the elements
\begin{itemize}
\item[]
  $(A\otimes A)\left((\sqrtpara
    e_1^{(2,\para)}+e_3^{(2,\para)})\otimes
    f_3^{(2,\para)})+(\sqrtpara
    e_1^{(2,\para)}-e_3^{(2,\para)})\otimes f_3^{(2,\para)}\right)$ and
\item[] \begin{multline*}(A\otimes A)\left((\para e_1^{(2,\para)}\otimes f_1^{(2,\para)}+2\sqrtpara e_2^{(2,\para)}\otimes f_2^{(2,\para)}+ e_3^{(2,\para)}\otimes f_3^{(2,\para)})-\right.\\
    \left.(\para e_1^{(2,\para)}\otimes f_1^{(2,\para)}-2\sqrtpara
      e_2^{(2,\para)}\otimes f_2^{(2,\para)}+ e_3^{(2,\para)}\otimes
      f_3^{(2,\para)})\right)
\end{multline*}
\end{itemize}
belong to $\ker\tilde{\phi}_{3,\para'}$. Applying
$\tilde{\phi}_{3,\para'}$ and using the fact that the basis elements
$g_k^{(3,\para')}$ for $k=1,2,3$ are linearly independent, we derive
\[
\sqrtpara\varepsilon^2+\para'^{\frac{1}{2}}=0, \quad
\sqrtpara\varepsilon^6+\para'^{\frac{1}{2}}=0, \quad
\para\varepsilon^4-\para'=0.
\]
It follows that $\varepsilon^4=1$, that
$\para'^{\frac{1}{2}}=-\varepsilon^2\para^{\frac{1}{2}}$, and that
$\para=\para'$. We conclude observing that, if $\varepsilon^2=1$, then
we get $\para'^{\frac{1}{2}}=-\para^{\frac{1}{2}}$ while, if
$\varepsilon$ is a primitive $4$th root of unity, then
$\para'^{\frac{1}{2}}=\para^{\frac{1}{2}}$.
\end{proof}

\begin{proposition}\label{prop:same i}
  The following are equivalent:
\begin{enumerate}[label=$(\arabic*)$]
\item The Lie algebras $\la{i,\para}{\alg}$ and $\la{i,\para'}{\alg}$
  are isomorphic.
\item The equality $\para=\para'$ holds in $\alg$ and, if $i\in\graffe{2,3}$, then either  
  \begin{enumerate}[label=$(\alph*)$]
  \item $\sqrtpara=\para'^{\frac{1}{2}}$ or
  \item $\sqrtpara=-\para'^{\frac{1}{2}}$ and $\alg$ contains a primitive $4$th root of unity.
\end{enumerate}
\end{enumerate}
\end{proposition}

\begin{proof}
$(2)\Rightarrow(1)$ If $i=1$, the statement is clear. Assume now that $i\in\graffe{2,3}$. If $\sqrtpara=\para'^{\frac{1}{2}}$, then $\tuB_{i,\para}=\tuB_{i,\para'}$ implying that $\la{i,\para}{\alg}\cong\la{i,\para'}{\alg}$. Assume now that $\sqrtpara=-\para'^{\frac{1}{2}}$ and that $\alg$ contains a primitive $4$th root of unity. Then $\tuB_{2,\para}=\tuB_{3,\para'}$ and so Proposition~\ref{prop:i<j} yields $\la{2,\para}{\alg}\cong\la{3,\para'}{\alg}\cong \la{2,\para'}{\alg}\cong\la{3,\para}{\alg}$.

$(1)\Rightarrow(2)$ Assume $\la{i,\para}{K}\cong \la{i,\para'}{K}$ and let
  $\beta=\diag(A,A,A_T):\la{i,\para}{K}\rightarrow\la{i,\para'}{K}$ be
  an isomorphism as in Lemma~\ref{lemma:tentative}. Then there exist
  $\varepsilon\in K$ such that $\sqrtpara=\varepsilon^2\para'^{\frac{1}{2}}$ and, up to
  a scalar multiple, the matrix $A$ equals
  $\diag(\varepsilon^2, \varepsilon^3,1)$;
%\[
%A=\begin{pmatrix}
%\varepsilon^2 & 0 & 0 \\
%0 & \varepsilon^3 & 0 \\
%0 & 0 & 1
%\end{pmatrix};
%\] 
  see Lemma~\ref{lemma:Aunique}. We fix such $\varepsilon$ and proceed
  by considering the cases $i=1$ and $i\in\{2,3\}$ separately.

  If $i=1$, then $g_1^{(1,\para)}=[e_1^{(1,\para)},f_3^{(1,\para)}]=-[e_2^{(1,\para)},f_2^{(1,\para)}]$ in $\la{1,\para}{K}$, whence
\[
\varepsilon^2g_1^{(1,\para')}=\beta([e_1^{(1,\para)},f_3^{(1,\para)}])=-\beta([e_2^{(1,\para)},f_2^{(1,\para)}])=-\varepsilon^6(-g_1^{(1,\para')}).
\]
As a result, $\varepsilon^4=1$ and so $\para=\para'$.

If $i\in\graffe{2,3}$, then
$-\sqrtpara g_2^{(i,\para)}=\sqrtpara[e_1^{(i,\para)},f_2^{(i,\para)}]=\pm [e_2^{(i,\para)},f_3^{(i,\para)}]$,% in $\la{i,\para}{K}$, 
whence
\[
\sqrtpara\varepsilon^5(-g_2^{(i,\para')})=\sqrtpara\beta([e_1^{(i,\para)},f_2^{(i,\para)}])=\pm \beta([e_2^{(i,\para)},f_3^{(i,\para)}])=\pm \varepsilon^3(\mp \para'^{\frac{1}{2}} g_2^{(i,\para')})=-\varepsilon^3\para'^{\frac{1}{2}} g_2^{(i,\para')}.
\]
It follows that $\varepsilon^4=1$, which yields $\para=\para'$ and, in particular, $\sqrtpara=\para'^{\frac{1}{2}}$ or $\sqrtpara=-\para'^{\frac{1}{2}}$. If $\sqrtpara=-\para'^{\frac{1}{2}}$, then Proposition~\ref{prop:i<j} yields that $K$ possesses a primitive $4$th root of unity. 
\end{proof}

\begin{theorem}\label{th:last}
  The $K$-Lie algebras $\la{i,\para}{\alg}$ and $\la{j,\para'}{\alg}$ are
  isomorphic if and only if $\para=\para'$ in $\alg$ and either
\begin{enumerate}[label=$(\arabic*)$]
\item $i=j$ and, if $i\in\graffe{2,3}$, then either
\begin{enumerate}[label=$(1.\alph*)$]
  \item $\sqrtpara=\para'^{\frac{1}{2}}$ or
  \item $\sqrtpara=-\para'^{\frac{1}{2}}$ and $\alg$ contains a
    primitive $4$th root of unity or
\end{enumerate}
\item $\graffe{i,j}=\graffe{2,3}$ and either
\begin{enumerate}[label=$(2.\alph*)$]
  \item $\sqrtpara=-\para'^{\frac{1}{2}}$ or
  \item $\sqrtpara=\para'^{\frac{1}{2}}$ and $\alg$ contains a
    primitive $4$th root of unity.
\end{enumerate}
\end{enumerate}
\end{theorem}

\begin{proof}
  Combine Propositions~\ref{prop:i<j} and~\ref{prop:same i}.
\end{proof}

\noindent
Theorem~\ref{th:iso} is the special case of Theorem~\ref{th:last} for
$\alg=\Fp$, via the Baer correspondence.

We conclude the paper by proving Theorem~\ref{th:intro}.  The groups
$\gp{i}{F}$ are defined by solving the equation \eqref{eq:gen.hess}
associated to the Weierstrass form of $E$ and then applying the
construction from Section~\ref{subsec:groups}. We have exactly three
solutions to \eqref{eq:gen.hess}, corresponding to the three
nontrivial $2$-torsion points of $E$, thanks to
\cite[Thm.~1]{RaviTri14} and Theorem~\ref{th:intro}(1) is clear from
the construction of $\gp{i}{F}$. Theorem~\ref{th:intro}(2) is a result
of the combination of Lemma~\ref{lemma:HessianAB},
Corollary~\ref{cor:formula}, and Corollary~\ref{lemma:galois}. The
rest of Theorem~\ref{th:intro} is given by
Proposition~\ref{prop:lifting IsomE.para} and Theorem~\ref{th:iso}.

\begin{acknowledgements}
  We are very grateful to Alex Galanakis, Avinash Kulkarni, Josh
  Maglione, Eamonn O'Brien, Uri Onn, Giulio Orecchia, Tobias Rossmann,
  Michael Vaughan-Lee, and James Wilson for mathematical
  discussions. We thank the anonymous referee for comments that helped
  us improve this paper's exposition.
\end{acknowledgements}

\providecommand{\bysame}{\leavevmode\hbox to3em{\hrulefill}\thinspace}
\providecommand{\MR}{\relax\ifhmode\unskip\space\fi MR }
\providecommand{\MRhref}[2]{%
  \href{http://www.ams.org/mathscinet-getitem?mr=#1}{#2}
}
\providecommand{\href}[2]{#2}

\vspace*{2em}
\noindent
{\footnotesize
\begin{minipage}[t]{0.48\textwidth}
  Mima Stanojkovski\\
  Max-Planck-Institute for Mathematics in the Sciences\\
  Inselstrasse 22\\
  04103 Leipzig\\
  Germany\\
  \quad\\
  E-mail: \href{mailto:mima.stanojkovski@mis.mpg.de}{mima.stanojkovski@mis.mpg.de}
\end{minipage}
\hfill
\begin{minipage}[t]{0.28\textwidth}
  Christopher Voll\\
  Fakult\"at f\"ur Mathematik\\
  Universit\"at Bielefeld\\
  D-33501 Bielefeld\\
  Germany\\
  \quad\\
  E-mail: \href{mailto:C.Voll.98@cantab.net}{C.Voll.98@cantab.net}
\end{minipage}
}

\end{document}